\documentclass[12pt,english]{article}
\usepackage{amsmath}
\usepackage{graphicx}
\usepackage{url} 
\usepackage{color}
\usepackage{babel}
\usepackage{float}
\usepackage{mathtools}
\usepackage{amsmath}
\usepackage{amsthm}
\usepackage{amssymb}
\usepackage{graphicx}
\usepackage{setspace}
\usepackage{nicefrac}
\usepackage{csquotes}
\usepackage[backend=biber, sorting=none, maxcitenames=2, style=numeric]{biblatex}
\providetoggle{blx@lang@captions@<language>}
\usepackage[bookmarks]{hyperref}
\hypersetup{
	unicode=false,
	pdftoolbar=true,
	pdfmenubar=true,
	pdffitwindow=false,     
	pdfstartview={FitH},    
	pdftitle={nonconvexAG},    
	pdfauthor={Kai Yang},     
	pdfsubject={Subject},   
	pdfcreator={Kai Yang},   
	pdfproducer={Kai Yang}, 
	pdfkeywords={}, 
	pdfnewwindow=true,      
	colorlinks=true,       
	linkcolor=red,          
	citecolor=blue,        
	filecolor=black,      
	urlcolor=cyan           
}

\newcommand{\blind}{0}

\PassOptionsToPackage{natbib=true}{biblatex}

\makeatletter

\providecommand{\\}{\\}
\floatstyle{ruled}
\newfloat{algorithm}{tbp}{loa}
\providecommand{\algorithmname}{Algorithm}
\floatname{algorithm}{\protect\algorithmname}

\theoremstyle{plain}
\newtheorem{thm}{\protect\theoremname}
\theoremstyle{plain}
\newtheorem{cor}{\protect\corollaryname}
\theoremstyle{plain}
\newtheorem{lem}{\protect\lemmaname}

\usepackage{xcolor}
\usepackage{multirow}\usepackage{colonequals}
\usepackage{amsfonts}\usepackage{tocbibind}\usepackage{bm}\usepackage{bbm}\usepackage{commath}\usepackage{multicol}\allowdisplaybreaks

\usepackage{algorithm, algorithmic}

\providecommand{\corollaryname}{Corollary}
\providecommand{\lemmaname}{Lemma}
\providecommand{\theoremname}{Theorem}

\bibliography{./nonconvexAG.bib}

\begin{document}

\def\spacingset#1{\renewcommand{\baselinestretch}%
{#1}\small\normalsize} \spacingset{1}

\if0\blind
{
  \title{\bf Accelerated Gradient Methods for Sparse Statistical Learning with Nonconvex Penalties}
  \author{Kai Yang\\
    Department of Epidemiology, Biostatistics and \\Occupational Health, McGill University\\
    and \\
    Masoud Asgharian\\
    Department of Mathematics and Statistics, McGill University\\
    and \\
    Sahir Bhatnagar\\
    Department of Epidemiology, Biostatistics and \\Occupational Health, McGill University
    }
  \maketitle
} \fi

\if1\blind
{
  \bigskip
  \bigskip
  \bigskip
  \begin{center}
    {\LARGE\bf Accelerated Gradient Methods for Sparse Statistical Learning with Nonconvex Penalties}
\end{center}
  \medskip
} \fi

\bigskip
\begin{abstract} 

Nesterov's accelerated gradient (AG) is a popular technique to optimize 
objective functions comprising two components: a convex loss and a penalty function. While AG methods perform well for convex penalties, such as the LASSO,  convergence issues may arise when it is applied to nonconvex penalties, such as SCAD. 
A recent proposal generalizes Nesterov's AG method to the nonconvex setting. 
The proposed algorithm requires specification of  several hyperparameters for its practical 
application. 
Aside from some general conditions, there is no explicit rule for selecting the hyperparameters, and how different selection can affect convergence of the algorithm.  
In this article, we propose a hyperparameter setting based on the complexity upper bound to accelerate convergence, and consider the application of this nonconvex AG algorithm to high-dimensional linear and logistic sparse learning problems.
We further establish the rate of convergence and present a simple and useful bound  { to characterize our proposed optimal} damping  sequence. Simulation studies show that convergence can be made, on average, considerably faster than that of the conventional proximal gradient algorithm. Our experiments also show that the proposed method generally outperforms the current state-of-the-art methods in terms of signal recovery. 
\end{abstract}

\noindent%

\spacingset{1.45}

\section{Introduction}


Sparse learning is an important component of modern data science and is an essential tool for the statistical analysis of high-dimensional data, with significant applications in signal processing and statistical genetics, among others. Penalization
is commonly used to achieve sparsity in parameter estimation. The prototypical optimization problem for obtaining penalized estimators is 
\[
\hat{\boldsymbol{\beta}}\in\arg\min_{\boldsymbol{\beta}\in\mathbb{R}^{q+1}}\left[f\left(\boldsymbol{\beta}\right)+\sum_{j=1}^{q}p_{\lambda}\left(\beta_{j}\right)\right],
\]
where $f:\mathbb{R}^{q+1}\mapsto\mathbb{R}$ is a convex loss function,  $p_\lambda:\mathbb{R}\mapsto\mathbb{R}_{\geq 0}$ constitutes the penalty term, and $\lambda>0$ is the tuning parameter for the penalty. 
Commonly used penalization methods for sparse learning include:
LASSO (Least Absolute Shrinkage and Selection Operator)~\parencite{Tibshirani1996}, Elastic Net~\parencite{Zou2005}, SCAD (Smoothly Clipped Absolute Deviation)~\parencite{Fan2001} and MCP (Minimax Concave Penalty)~\parencite{Zhang2010}.
Among these penalties, parameter estimation with
SCAD and MCP leads to a nonconvex objective function. The nonconvexity poses a challenge in statistical computing, as most methods developed for convex objective functions might not converge when applied to the nonconvex counterpart. 


Various approaches have been proposed to carry out
parameter estimation with SCAD or MCP penalties. \citeauthor{Zou2008}~\parencite*{Zou2008} proposed a local linear approximation, which yields a first-order majorization-minimization (MM) algorithm.  \citeauthor{Kim2008}~\parencite*{Kim2008} discussed a difference-of-convex
programming (DCP) method for ordinary least square estimators penalized by the SCAD penalty,
which was later generalized by \citeauthor{Wang2013}~\parencite*{Wang2013} to a general class of nonconvex
penalties to produce a first-order algorithm. These first-order methods belong to the class of proximal gradient descent methods, which are usually inefficient as relaxation is often expensive~\parencite{Nesterov2004}. The objective function is often ill-conditioned for sparse learning problems, and gradient descent with constant step size is especially inefficient for high-dimensional problems. Indeed, previous studies have suggested that the condition number of a square random matrix grows linearly with respect to its dimension~\parencite{Edelman1988}. Therefore, high-dimensional problems have a large condition number with high probability. Specific to gradient descent with constant step size, the trajectory will oscillate in the directions with a large eigenvalue, moving very slowly toward the directions with a small eigenvalue, making the algorithm inefficient.
\citeauthor{Lee2016}~\parencite*{Lee2016} developed a modified second-order method
originally designed for the ordinary least square loss function penalized by LASSO with extensions
to SCAD and MCP; this attempt was later extended to generalized linear
models, such as logistic and Poisson regression, and Cox's proportional hazard model. Quasi-Newton methods, or a mixture of first and second-order descent methods, have also been applied on nonconvex penalties~\parencite{Ibrahim2012,Ghosh2016}. However, for high-dimensional problems, these second-order methods are slow due to the computational cost of evaluating the secant condition. Concurrently, most first and second-order methods discussed above require a line-search procedure at each step to ensure global convergence, which is prohibitive when the number of parameters to estimate grows large. 
\citeauthor{Breheny2011}~\parencite*{Breheny2011} implemented a coordinate descent method in the {\tt ncvreg} {\tt R} package to carry out estimation for linear models with least squares loss or logistic regression, penalized by SCAD and MCP. \citeauthor{Mazumder2011}~\parencite*{Mazumder2011} also implemented a coordinate descent method in the {\tt sparsenet} {\tt R} package, which carries out a closed-form root-finding update in a coordinate-wise manner for penalized linear regression. Similar to how ill-conditioning makes gradient descent inefficient, coordinate descent methods are generally inefficient when the covariate correlations are high~\parencite{Friedman2007}. Previous studies have also found that coordinate-wise minimization might not converge for some nonsmooth objective functions~\parencite{Spall2012}. 
Furthermore, it is naturally challenging to run coordinate-wise minimization in parallel, as the algorithm must run in a sequential coordinate manner. 


Due to the low computational
cost and adequate memory requirement per iteration, first-order methods
without a line search procedure have become the primary approach for
high-dimensional problems arising from various areas~\parencite{Beck2017}. 
For smooth convex objective functions, Nesterov proposed the {\em accelerated
gradient method} (AG) to improve the rate of convergence from $O(1/N)$ for gradient descent
to $O(1/N^{2})$ while achieving global convergence~\parencite{Nesterov1983}.
Subsequently, Nesterov extended AG to composite convex problems~\parencite{Nesterov2012},
whereas the objective is the sum of a smooth convex function and a simple nonsmooth convex function. With proper step-size choices, Nesterov's
AG was later shown optimal to solve both smooth and nonsmooth
convex programming problems~\parencite{Lan2011}. 


Given that sparse learning problems are often high-dimensional,
Nesterov's AG has been frequently used for \emph{convex} problems
in statistical machine learning (e.g.,~\cite{Simon_2013,Yang_2014,Yu_2015,Akyildiz_2021}).
However, convergence is questionable if the convexity assumption is violated. 
Recently, \citeauthor{Ghadimi2015}~\parencite*{Ghadimi2015} generalized the AG method to nonconvex objective functions, hereafter referred to as the nonconvex AG method, and derived the rates of convergence for both smooth and composite objective functions. While this method can be applied to nonconvex sparse learning problems, several hyperparameters must be set prior to running the algorithm and can be difficult to choose in practice. Indeed, the nonconvex AG method has never been applied in the context of sparse statistical learning problems with nonconvex penalties, such as SCAD and MCP. 


This manuscript presents a detailed analysis of the complexity upper bound of the nonconvex AG algorithm and proposes a hyperparameter setting to accelerate convergence (Theorem \ref{thm:convex}). We further establish the rate of convergence (Theorem \ref{thm:1overk}) and present a simple and useful bound to characterize our proposed optimal damping sequence (Theorem \ref{thm:alpha-k-vanishing} and Corollary \ref{cor:alpha-vanishing}). Our simulation studies on penalized linear and logistic models show that the nonconvex AG method with the proposed hyperparameter selector converges considerably faster than other first-order methods. We also compare the signal recovery performance of the algorithm to that of {\tt ncvreg}, the state-of-the-art method based on coordinate descent, showing that the proposed method outperforms the state-of-the-art coordinate descent method. 


The rest of this manuscript is organised as follows. In Sections \ref{sec:Accelerated-Gradient-Method}, \ref{sec:aga}, \ref{sec:theory},
we will present an analysis of the nonconvex AG algorithm by \citeauthor{Ghadimi2015} to illustrate the algorithm as a generalization of Nesterov's AG. We also present formal results about the effect of hyperparameter settings on the complexity upper bound. Section \ref{sec:Simulation-Studies} will include simulation studies for linear and logistic models penalized by SCAD and MCP penalties. The simulation studies show that i) The AG method using our proposed hyperparameter settings converges faster than commonly used first-order methods for data with various $q/n$ and covariate correlation settings; and ii) our method outperforms the current state-of-the-art method, i.e. {\tt ncvreg}, in terms of signal recovery performance, especially when the signal-to-noise ratios are low. The proofs for the theorems are included in the Appendix \ref{sec:proof}.

\section{Motivation and Setup \label{sec:Accelerated-Gradient-Method}} 


Having built on Nesterov's seminal work, \citeauthor{Ghadimi2015}~\parencite*{Ghadimi2015} 
considered the following composite optimization problem: 
\begin{equation}
    \min_{x\in\mathbb{R}^{q+1}}\Psi\left(x\right)+\chi\left(x\right),\ \Psi\left(x\right)\coloneqq f\left(x\right)+h\left(x\right),
\tag{$\mathcal{P}$}\label{eq:optproblem}
\end{equation}
where $f\in\mathcal{C}_{L_{f}}^{1,1}(\mathbb{R}^{q+1},\mathbb{R})$ is convex, 
$h\in\mathcal{C}_{L_{h}}^{1,1}(\mathbb{R}^{q+1},\mathbb{R})$ is possibly nonconvex, and $\chi$ is a convex function over a bounded domain, and $\mathcal{C}_{L}^{1,1}$ denotes the class of first-order Lipschitz smooth functions with $L$ being the Lipschitz constant. They devised Algorithm \ref{alg:Accelerated-Gradient-Algorithm} discussed in details in next section, and presented a theoretical analysis of their algorithm. 

Some commonly used nonconvex penalties, such as SCAD and MCP, have a form that can naturally be decomposed into summation of a convex and a nonconvex function satisfying the conditions required by \citeauthor{Ghadimi2015}~\parencite*{Ghadimi2015}. When such penalties are added to a smooth convex deviance measure, such as negative of typical log-likelihoods, the resulting optimization problem follows the form of optimization problem \ref{eq:optproblem}. As we show below this is, in particular, the case when the deviance measure is a quadratic loss and the penalty is either SCAD or MCP. The quadratic loss plays the role of $f$. The other two functions, i.e. $h$ and $\chi$ are specified for both SCAD and MCP penalties. Define 


\begin{equation}p_{\lambda,a,\text{SCAD}}\left(\boldsymbol{\beta}\right)=\chi\left(\boldsymbol{\beta}\right)+h_{\text{SCAD}}\left(\boldsymbol{\beta}\right),
\label{eq:SCAD}
\end{equation}
\begin{equation}p_{\lambda,\gamma,\text{MCP}}\left(\boldsymbol{\beta}\right)=\chi\left(\boldsymbol{\beta}\right)+h_{\text{MCP}}\left(\boldsymbol{\beta}\right);
\label{eq:MCP}
\end{equation}
where $\boldsymbol{\beta}\coloneqq\left[\beta_{0},\beta_{1},\dots,\beta_{q}\right]^{T}$, $\chi\left(\boldsymbol{\beta}\right)=\sum_{j=1}^{q}\lambda\lvert\beta_{j}\rvert$, and 
\begin{align}h_{\text{SCAD}}\left(\boldsymbol{\beta}\right)=\sum_{j=1}^{q}\begin{cases}
0; & \lvert\beta_{j}\rvert\leq\lambda\\
\frac{2\lambda\lvert\beta_{j}\rvert-\beta_{j}^{2}-\lambda^{2}}{2\left(a-1\right)}; & \lambda<\lvert\beta_{j}\rvert<a\lambda\\
\frac{1}{2}\left(a+1\right)\lambda^{2}-\lambda\lvert\beta_{j}\rvert; & \lvert\beta_{j}\rvert\geq a\lambda
\end{cases} & \in\mathcal{C}_{L_{\text{SCAD}}}^{1,1}\label{eq:SCAD-Lsmoothness}\\
h_{\text{MCP}}\left(\boldsymbol{\beta}\right)=\sum_{j=1}^{q}\begin{cases}
-\frac{\beta_{j}^{2}}{2\gamma}; & \lvert\beta_{j}\rvert<\gamma\lambda\\
\frac{1}{2}\gamma\lambda^{2}-\lambda\lvert\beta_{j}\rvert; & \lvert\beta_{j}\rvert\geq\gamma\lambda
\end{cases} & \in\mathcal{C}_{L_{\text{MCP}}}^{1,1}\label{eq:MCP-Lsmoothness}
\end{align}
In the above equations, $\lambda>0,a>2,\gamma>1$ are the penalty tuning parameters. It is trivial that, in \eqref{eq:SCAD} and \eqref{eq:MCP},
$\chi\left(\boldsymbol{\beta}\right)$ is convex and the remaining term is
a first-order smooth concave function. 
In view of the optimization problem \ref{eq:optproblem}, when applying SCAD/MCP on a convex $\mathcal{C}_{L_{\ell}}^{1,1}$ statistical learning objective function, $f=-2\ell$ will be the convex component; $h_{\text{SCAD}},h_{\text{MCP}}$ will be the smooth nonconvex component with $L_{SCAD}=\frac{1}{a-1}$ and $L_{MCP}=\frac{1}{\gamma}$; and $\chi=\sum_{j=1}^{q}\lambda\lvert\beta_j\rvert$ will be the nonsmooth convex component. For high-dimensional statistical learning problems, 
the L-smoothness constant for the smooth nonconvex component, $L_{SCAD}$ and $L_{MCP}$, are often negligible when compared to the greatest singular value of the design matrix~\parencite{Meckes2021}. In statistical learning applications, most unconstrained problems can, in fact, be reduced to problems over a bounded domain, as information often suggests the boundedness of the variables.

\section{The Accelerated Gradient Algorithm} \label{sec:aga}
This Section comprises two subsections. Subsection \ref{sec:nagm} includes an algorithm proposed by \citeauthor{Ghadimi2015}~\parencite*{Ghadimi2015} for solving the composite optimization problem \ref{eq:optproblem}. In Subsection \ref{sec:hnagm} we propose an approach for selecting the hyperparameters of the algorithm by minimizing the complexity upper bound \eqref{eq:complexity-bound}

\subsection{Nonconvex Accelerated Gradient Method} \label{sec:nagm}

Building on Nesterov's AG algorithm, \citeauthor{Ghadimi2015}~\parencite*{Ghadimi2015} proposed the following algorithm for solving the composite optimization problem \ref{eq:optproblem}. 

\begin{algorithm}[H] 
\caption{Accelerated Gradient Algorithm\label{alg:Accelerated-Gradient-Algorithm}}
\begin{algorithmic}
\REQUIRE {starting point $x_0\in \mathbb{R}^{q+1}$, $\{\alpha_k\}$ s.t. $\alpha_1=1$ and $\forall k\geq 2, 0<\alpha_k<1$, $\{\omega_k>0\}$, and $\{\delta_k>0\}$} 		\STATE {0. Set $x_0^{ag}=x_0$ and $k=1$} 		\STATE {1. Set \begin{equation} 			\label{eqn:x-md} 				x_{k}^{md}=\alpha_{k}x_{k-1}^{ag}+\left(1-\alpha_{k}\right)x_{k-1} 			\end{equation} } 		\STATE {2. Compute $\nabla\Psi\left(x_{k}^{md}\right)$ and set \begin{align} 				x_{k}= & x_{k-1}-\delta_{k}\nabla\Psi\left(x_{k}^{md}\right) \text{ (smooth)}& x_{k}= & \mathcal{P}\left(x_{k-1},\nabla\Psi\left(x_{k}^{md}\right),\delta_{k}\right) \text{ (composite)}\label{eqn:x-k}\\ 				x_{k}^{ag}= & x_{k}^{md}-\omega_{k}\nabla\Psi\left(x_{k}^{md}\right) \text{ (smooth)}& x_{k}^{ag}= & \mathcal{P}\left(x_{k}^{md},\nabla\Psi\left(x_{k}^{md}\right),\omega_{k}\right) \text{ (composite)}\label{eqn:x-ag} 		\end{align}} 		\STATE {3. Set $k=k+1$ and go to step 1} 		\ENSURE {Minimizer $x_N^{md}$} 	\end{algorithmic}
\end{algorithm}

In Algorithm \ref{alg:Accelerated-Gradient-Algorithm}, ``smooth''
represents the updating formulas for smooth problems, and ``composite''
represents the update formulas for composite problems, and $\mathcal{P}$
is the proximal operator defined as: 
\[
\mathcal{P}\left(x,y,c\right)\coloneqq\arg\min_{u\in\mathbb{R}^{q+1}}\left\{ \left\langle y,u\right\rangle +\frac{1}{2c}\left\Vert u-x\right\Vert ^{2}+\chi\left(u\right)\right\} .
\]
It is evident that the composite counter-part of the algorithm is
the Moreau envelope smoothing of the simple nonconvex function; for
this reason, in later analysis of the algorithm, we will use smooth
updating formulas for the sake of parsimony. As an interpretation
of the algorithm, $\left\{ \alpha_{k}\right\} $ controls the damping
of the system, and $\omega_{k}$ controls the step size for the ``gradient
correction'' update for momentum method. In what follows, $\Gamma_{k}$ is defined recursively as:
\[
\Gamma_{k}\coloneqq\begin{cases}
1, & k=1;\\
\left(1-\alpha_{k}\right)\Gamma_{k-1}, & k\geq2.
\end{cases}
\] 
\citeauthor{Ghadimi2015}~\parencite*{Ghadimi2015} proved that under the following conditions: 
\begin{align}
 & \alpha_{k}\delta_{k}\leq\omega_{k}<\frac{1}{L_{\Psi}},\ \forall k=1,2,\dots N-1\text{ and}\label{eq:convcond1}\\
 & \frac{\alpha_{1}}{\delta_{1}\Gamma_{1}}\geq\frac{\alpha_{2}}{\delta_{2}\Gamma_{2}}\geq\cdots\geq\frac{\alpha_{N}}{\delta_{N}\Gamma_{N}},\label{eq:convcond2}
\end{align}
the rate of convergence for composite optimization problems can be
illustrated by the following complexity upper bound: 
\begin{align}
 & \min_{k=1,\dots,N}\left\Vert \mathcal{G}\left(x_{k}^{md},\nabla\Psi\left(x_{k}^{md}\right),\omega_{k}\right)\right\Vert ^{2}\nonumber \\
 & \hspace{1.2 in} \leq \left[\sum_{k=1}^{N}\Gamma_{k}^{-1}\omega_{k}\left(1-L_{\Psi}\omega_{k}\right)\right]^{-1}\left[\frac{\left\Vert x_{0}-x^{*}\right\Vert ^{2}}{\delta_{1}}+\frac{2L_{h}}{\Gamma_{N}}\left(\left\Vert x^{*}\right\Vert ^{2}+M^{2}\right)\right].\label{eq:complexity-bound}
\end{align}
In the above inequality, $\mathcal{G}\left(x_{k}^{md},\nabla\Psi\left(x_{k}^{md}\right),\omega_{k}\right)$
is the analogue to the gradient for smooth functions defined
by: 
\[
\mathcal{G}\left(x,y,c\right)\coloneqq\frac{1}{c}\left[x-\mathcal{P}\left(x,y,c\right)\right].
\]
In accelerated gradient settings, $x$ corresponds to the past iteration,
$y$ corresponds to the smooth gradient at $x$, and $c$ corresponds
to the step size taken. 

\subsection{Hyperparameters for Nonconvex Accelerated Gradient Method} \label{sec:hnagm}

Here we discuss how hyperparameters, $\alpha_k$, $\omega_k$ and $\delta_k$ can be selected to accelerate convergence of Algorithm \ref{alg:Accelerated-Gradient-Algorithm} by minimizing the complexity upper bound. 
From Lemma \ref{lem:convergence-cond-meaning}, it is clear that the conditions \eqref{eq:convcond1} and \eqref{eq:convcond2} merely present a lower
bound for the vanishing rate of $\left\{ \alpha_{k}\right\} $. We
also observe that the right-hand side of \eqref{eq:bifurcation} is
monotonically increasing with respect to $\alpha_{k}$; thus, to obtain
the maximum values for $\left\{ \alpha_{k}\right\} $, it is sufficient
to maximize $\alpha_{k}$ recursively. 

Using \eqref{eqn:x-md}, \eqref{eqn:x-k}, and \eqref{eqn:x-ag}, we have 
\begin{align*}
\frac{x_{k+1}^{md}-\left(1-\alpha_{k+1}\right)x_{k}^{ag}}{\alpha_{k+1}}= & \frac{x_{k}^{md}-\left(1-\alpha_{k}\right)x_{k-1}^{ag}}{\alpha_{k}}-\delta_{k}\nabla\Psi\left(x_{k}^{md}\right)\text{ and}\\
x_{k}^{ag}= & x_{k}^{md}-\omega_{k}\nabla\Psi\left(x_{k}^{md}\right).
\end{align*}
By sorting out the terms in the above equations, we obtain the following
updating formulas: 
\begin{align}
x_{k}^{ag}= & x_{k}^{md}-\omega_{k}\nabla\Psi\left(x_{k}^{md}\right)\label{eq:grad-correction}\\
x_{k+1}^{md}= & x_{k}^{ag}+\alpha_{k+1}\cdot\left(\frac{1}{\alpha_{k}}-\frac{\delta_{k}}{\omega_{k}}\right)\cdot\left(\omega_{k}\nabla\Psi\left(x_{k}^{md}\right)\right)+\alpha_{k+1}\cdot\left(\frac{1}{\alpha_{k}}-1\right)\left(x_{k}^{ag}-x_{k-1}^{ag}\right)\label{eq:nonconvex-momentum}
\end{align}


Compared to Nesterov's AG, the AG method proposed by Ghadimi and
Lan differs by the convergence conditions \eqref{eq:convcond1} and
\eqref{eq:convcond2}, and the inclusion of the term $\alpha_{k+1}\cdot\left(\frac{1}{\alpha_{k}}-\frac{\delta_{k}}{\omega_{k}}\right)\cdot\left(\omega_{k}\nabla\Psi\left(x_{k}^{md}\right)\right)$
in \eqref{eq:nonconvex-momentum}. Since $\alpha_{k+1}\cdot\left(\frac{1}{\alpha_{k}}-\frac{\delta_{k}}{\omega_{k}}\right)\geq0$
is implied by convergence condition \eqref{eq:convcond1}, this added
term functions as a step to reduce the magnitude of ``gradient correction''
presented in \eqref{eq:grad-correction}: the resulting framework
will keep the same momentum compared to Nesterov's AG, but the momentum
step update will occur at a midpoint between $x_{k}^{ag}$ and $x_{k}^{md}$
to yield $x_{k+1}^{md}$. Such a framework suggests that the proposed
algorithm is merely a midpoint generalization in the gradient correction
step of Nesterov's AG. Therefore, \emph{the acceleration occurs to
the convex component $f$ of the objective function $\Psi$}. Following
this intuition, we proceed to investigate the optimization hyperparameter
settings for the most accelerating effect in Theorem \ref{thm:convex}
based on the idea of minimizing the complexity upper bound \eqref{eq:complexity-bound}
when the objective function is convex; i.e., when $h \equiv 0$. 

It can be deduced from \eqref{eq:bifurcation} that an increasing
sequence of $\left\{ \delta_{k}\right\} $ allows a slower vanishing
rate for $\left\{ \alpha_{k}\right\} $. Specifically, the existence
of $\delta_{1}$ in \eqref{eq:complexity-bound} can be explained
as the following: the momentum initialization step in Algorithm \ref{alg:Accelerated-Gradient-Algorithm}
indicates that $x_{1}^{md}=x_{0}^{ag}=x_{0}$. We also have $x_{1}^{ag}=x_{1}^{md}-\omega_{1}\nabla\Psi\left(x_{1}^{md}\right)=x_{0}^{ag}-\omega_{1}\nabla\Psi\left(x_{0}\right)$
for smooth problems or $x_{1}^{ag}=\mathcal{P}\left(x_{1}^{md},\nabla\Psi\left(x_{1}^{md}\right),\omega_{1}\right)=\mathcal{P}\left(x_{0}^{ag},\nabla\Psi\left(x_{0}\right),\omega_{1}\right)$
for composite problems. In view of \eqref{eq:nonconvex-momentum},
the momentum initializes as $x_{1}^{ag}-x_{0}^{ag}=-\omega_{1}\nabla\Psi\left(x_{0}\right)$
for smooth problems. Thus, should $\delta_{1}<\omega_{1}$
take a smaller value, $\alpha_{2}\cdot\left(\frac{1}{\alpha_{1}}-\frac{\delta_{1}}{\omega_{1}}\right)>0$;
i.e., $x_{2}^{md}$ is a convex combination of $x_{1}^{ag}$ and the
initial point $x_{0}$, and the smaller $\delta_{1}$ is, the closer
$x_{2}^{md}$ is to $x_{0}$. Meanwhile, a smaller $\delta_{1}$
allows a faster increasing sequence $\left\{ \delta_{k}\right\} $;
hence a slower-vanishing sequence $\left\{ \alpha_{k}\right\} $ can
be achieved to incorporate more momentum. This process can be interpreted as follows: when $x_{2}^{md}$ does not retain the full step update
from the initial point $x_{0}$, more initial momentum will be allowed
to accumulate, as the initial momentum is in the same direction as
the update. We therefore choose $\delta_{1}=\omega_{1}$;
i.e., to let $x_{2}^{md}$ retain fully the update from $x_{0}$ in
the direction of $-\omega_{1}\nabla\Psi\left(x_{0}\right)$, such that
no \emph{excess} initial momentum will be needed to account for initial
update deficiency in this direction.

\section{Theoretical Analysis of the Algorithm}\label{sec:theory}


For gradient methods without a line-search procedure, the
step size for the gradient correction is usually set to be a constant.
Based on this convention, we assume $\omega_{k}=\beta$ for $k=1,2,\dots,N$.
Theorem \ref{thm:convex} below presents the optimal choice of hyperparameters under mild conditions. 
\begin{thm}
\label{thm:convex} Assume 
conditions \eqref{eq:convcond1} and \eqref{eq:convcond2} hold. Let $\delta_{1}=\omega_{k}=\omega$ and $h=0$.
Then the complexity upper bound \eqref{eq:complexity-bound} is minimized by: 
\begin{align}
\bar{\alpha}_{k+1}= & \frac{2}{1+\sqrt{1+\frac{4}{\bar{\alpha}_{k}^{2}}}},\ \bar{\alpha}_{1}=1\label{eq:alpha-settings},\\
\bar{\delta}_{k+1}= & \frac{\bar{\omega}}{\bar{\alpha}_{k+1}}\label{eq:lambda-setting},\\
\bar{\omega}= & \frac{2}{3L_{\Psi}}\label{eq:beta-setting}.
\end{align}
\end{thm}
\begin{proof}
See Appendix \ref{sec:proof-thm-1}.
\end{proof}

As illustrated by the proof of the above theorem, the optimization hyperparameter
settings \eqref{eq:alpha-settings}, \eqref{eq:lambda-setting}, and
\eqref{eq:beta-setting} allow for the greatest values of $\left\{ \alpha_{k}\right\} $
under the constant gradient-correction step size and maximum initial
update assumptions; i.e., condition 1. Such settings allow the most
acceleration for the convex component. Although a greater momentum
will result in a much faster convergence at the initial stage of the
algorithm, it will also result in oscillations of larger magnitudes
near the minimizer. Therefore, in the following theorem, we will show
that the complexity upper bound will always maintain $O\left(1/N\right)$
rate of convergence. This observation implies that the accelerated gradient method's worst-case scenario is at least as good as $O\left(1/N\right)$ for gradient descent in terms of the rate of convergence. 

\begin{thm}
\label{thm:1overk} Assume conditions \eqref{eq:convcond1}
and \eqref{eq:convcond2} hold. Then under the assumptions of Theorem
\ref{thm:convex}
, the complexity upper bound is $O\left(1/N\right)$. 
\end{thm}
\begin{proof}
See Appendix \ref{sec:proof-thm-2}.
\end{proof}

The recursive formula for optimal momentum hyperparameter, $\left\{ \alpha_{k}\right\} $, as presented in \eqref{eq:alpha-settings},
is of a rather complicated structure. The next theorem illustrates the vanishing rate of $\left\{ \alpha_{k}\right\} $.  

\begin{thm}
\label{thm:alpha-k-vanishing} 
Let $\bar{\alpha}_{1}=1$ and \eqref{eq:alpha-settings} holds. Then
\begin{equation}
\label{eq:damping-bounds}
\frac{2}{\left(1+a\cdot k^{-b}\right)k+1}<\bar{\alpha}_{k}\leq\frac{2}{k+1}, \quad k=1,\dots,N, 
\end{equation}
for any $a>0,\ 0<b<1$, such that 
\begin{equation}
    \label{eq:k-vanishing-cond} a\left(1-b\right)\cdot2^{2-b}-ab\left(1-b\right)\cdot2^{-b}-1\geq0.
\end{equation}
\end{thm}
\begin{proof}
See Appendix \ref{sec:proof-thm-3}.
\end{proof}

The following corollary establishes a tight bound for the damping sequence, hence providing the speed of convergence of our proposed 
optimal damping sequence $\left\{ \bar{\alpha}_{k}\right\} $
to $\frac{2}{k+1}$. 

\begin{cor} \label{cor:alpha-vanishing}
The lower bound in \eqref{eq:damping-bounds} is maximized at
\[
\bar{a}_k=\frac{2^{\bar{b}_k}}{\left(1-\bar{b}_k\right)\left(4-\bar{b}_k\right)} \quad \text{and} \quad  
\bar{b}_k=\frac{2+5\left(\log\frac{2}{k}\right)+\sqrt{9\left(\log\frac{2}{k}\right)^{2}+4}}{2\left(\log\frac{2}{k}\right)}\label{optimal-b} \quad \text{for} \quad k \geq 8. 
\]
The lower bound in \eqref{eq:damping-bounds} therefore becomes 
\begin{equation}
\frac{k+1}{2}-\bar{\alpha}_{k}^{-1}=O\left(\log k\right)
\end{equation}
\end{cor}


\begin{proof}
See Appendix \ref{sec:proof-cor-1}.
\end{proof}

To better illustrate Corollary \ref{cor:alpha-vanishing}, we plot the value of $\log\left(\bar{a}_kk^{-b}\right)$ v.s. $(k,b)$ in Figure \ref{fig:numerical-plot}.
The plot shows that as $k$ grows large, the optimizer $\bar{b}_k$ converges to $1$ at a very slow rate. It also reflects on the speed of $1+\bar{a}_k\cdot k^{-\bar{b}_k}$, the coefficient of $k$ in the denominator of the lower bound in \eqref{eq:damping-bounds}, goes to $1$ as $k$ increases.

\begin{figure}[H]
\centering{}\includegraphics[scale=.8]{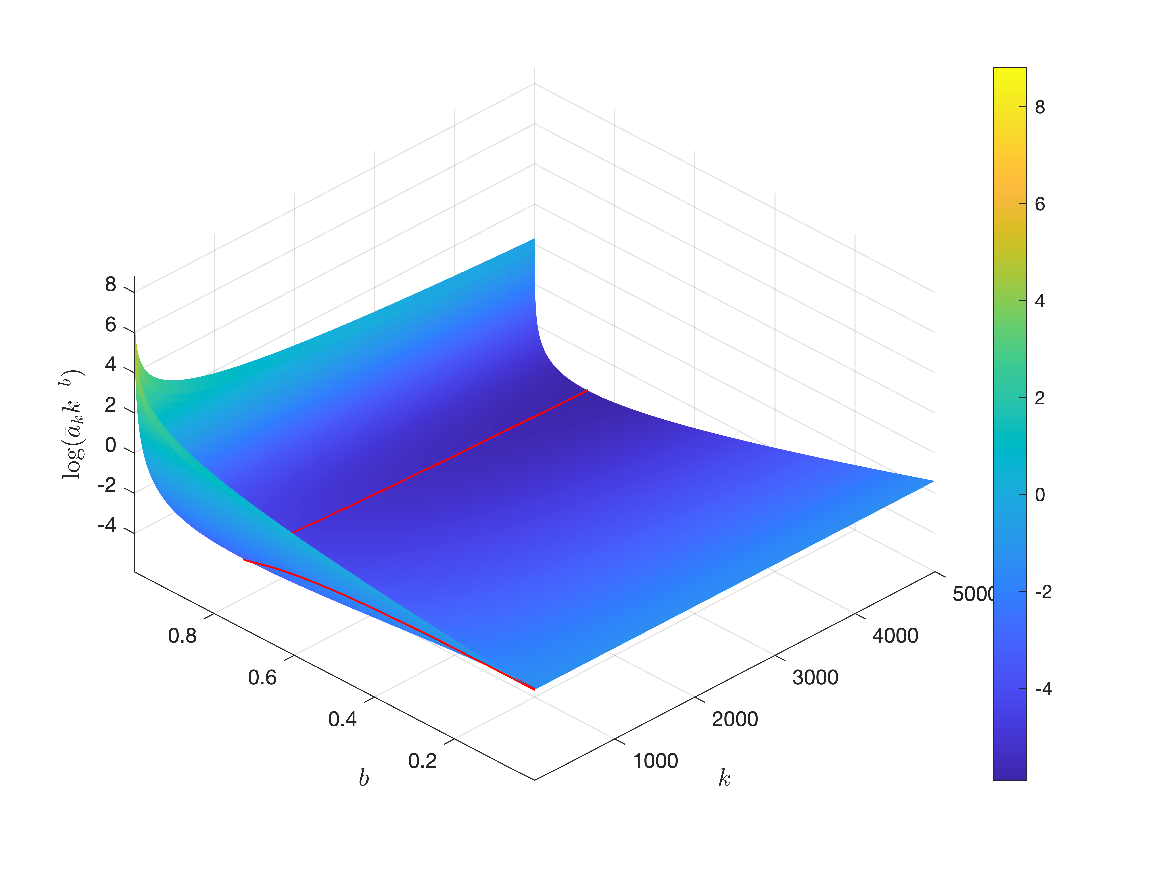}\caption{Numerical plots for Corollary \ref{cor:alpha-vanishing}. The figure plots $\log\left(\bar{a}_kk^{-b}\right)$ v.s. $k$ and $b$; the red line plots its minimizer $\bar{b}_k=\frac{2+5\left(\log\frac{2}{k}\right)+\sqrt{9\left(\log\frac{2}{k}\right)^{2}+4}}{2\left(\log\frac{2}{k}\right)}$ for each $k$. The plot reflects on the speed for the coefficient of $k$ in the denominator of the lower bound in \eqref{eq:damping-bounds} converges to $1$. The red line shows that $\bar{b}_k$ converges to $1$ at an extremely slow rate. \label{fig:numerical-plot}}
\end{figure}

\section{Simulations Studies\label{sec:Simulation-Studies} }


In this section, we conduct two sets of simulation studies for nonconvex penalized linear and logistic
models. We first visualize the convergence rates and signal recovery performance for each set of simulation studies using a single simulation replicate. Second, we compare the convergence rates across the first-order methods with varying $q/n$ ratios and covariate correlations for $100$ simulation replications. Lastly, we compare the signal recovery performance using our method to the state-of-the-art method, {\tt ncvreg}~\parencite{Breheny2011}, with varying covariate correlations and signal-to-noise ratios (SNRs) for $100$ simulation replications. Since the iterative complexity differs for the first-order methods and coordinate descent methods, the convergence rates in terms of the number of iterations are not directly comparable. { Thus, we choose to compare the computing time between AG, proximal gradient descent, and coordinate descent.}

\subsection{Simulation Setup} \label{sec:sim-setup}

Linear models with the OLS loss function is a popular method for modelling a continuous response. We aim to achieve signal recovery by solving the following problem for penalized linear models:
\[
\arg\min_{\boldsymbol{\beta}\in\mathbb{R}^{q+1}}\frac{1}{2n}\left\Vert \mathbf{X}\boldsymbol{\beta}-\mathbf{y}\right\Vert _{2}^{2}+\sum_{j=1}^{q}p_{\lambda}\left(\beta_{j}\right),
\]
where $p_\lambda:\mathbb{R}\mapsto\mathbb{R}_{\geq 0}$ is the SCAD or MCP penalty function. 
To compare the convergence rates across the first-order methods, we choose different $q/n$ ratios and the strength of correlation, $\tau$, between the covariates. These two parameters are most likely to impact the convergence rates. Median and corresponding $95\%$ bootstrap confidence intervals from $1000$ bootstrap replications for the number of iterations required for the iterative objective values to make a fixed amount of descent are reported. To compare the signal recovery performance between our AG method and the state-of-the-art package {\tt ncvreg}, we performed $100$ simulation replications with varying SNRs and covariate correlations, as they directly impact the signal recovery performance. The simulation studies we performed adapt the following setups:

\begin{itemize}
\item The total number of observations $n=1000$ for visualization plots and signal recovery performance comparison, and $n=200,500,1000,3000$ for convergence rate { and computing time} comparisons.
\item For visualization purposes, we perform one simulation replicate with the number of covariates $q=2004$, with $4$ nonzero signals being $2,-2,8,-8$. We perform $100$ simulation replications with the number of covariates $q = 2050$, with $5$ blocks of “true” signals equal-spaced with $500$ zeros in-between for convergence rate { and computing time comparison, as well as} signal recovery performance comparison. For each simulation replicate, the blocks of the ``true'' signals are simulated from $N_{10}\left(0.5,1\right)$, $N_{10}\left(5,2\right)$, $N_{10}\left(10,3\right)$, $N_{10}\left(20,4\right)$, $N_{10}\left(50,5\right)$, respectively. 
\item The design matrix, $\mathbf{X}$, is simulated from a multivariate Gaussian distribution
with mean $0$. The covariance matrix $\boldsymbol{\Sigma}$ is a $\tau-$Toeplitz matrix, where $\tau=0.5$ for the visualization plots and $\tau=0.1,0.5,0.9$ for the convergence rate { and computing time comparison, as well as} signal recovery performance comparison. All covariates are standardized; i.e., centered by the sample mean and scaled by the sample standard deviation.
\item The signal-to-noise ratio is set as $\text{SNR}=\frac{\sqrt{\boldsymbol{\beta}_{\text{true}}^{T}\boldsymbol{\Sigma}\boldsymbol{\beta}_{\text{true}}}}{\sigma}$, where $\boldsymbol{\beta}_{\text{true}}$ are the ``true'' coefficient values, and $\sigma$ is used as the residual standard deviation. $\text{SNR}=5$ for visualization plots, $\text{SNR}=3$ for convergence rate comparison, and $\text{SNR}=1,3,7,10$ for signal recovery performance comparison.
\item { For visualization plots, convergence rate and computing time comparisons,} we take $\lambda=0.5,a=3.7$ for SCAD and $\lambda=0.5,\gamma=3$ for MCP, unless otherwise specified. For signal recovery rate comparison, $\lambda$ sequence consists of
$50$ values equal-spaced from $\lambda_{\max}$\footnote{$\lambda_{\max}$ is the minimal value for $\lambda$ such that all
penalized coefficients are estimated as $0$.} to $0$. The tuning parameter $\lambda$ is chosen to minimize the (non-penalized) loss function value on a validation set of the same size as the training set.
\item { For signal recovery performance comparison, we use the same objective function as {\tt ncvreg} to ensure that the same value of penalty tuning parameters results in the same degree of penalization. We also adapt the same strong rule setup as {\tt ncvreg}~\parencite{Lee2015}.}
\end{itemize}

{ To compare the gradient-based methods and the coordinate descent method, we compare the computing time when both coded in {\tt Python/CuPy}. The coordinate descent method was coded based on the state-of-the-art pseudo-code~\parencite{Breheny2011}. All of the computing was carried out on a NVIDIA A100 GPU with CUDA compute capability of 8.0 on the Narval computing cluster from Calcul Qu\`ebec/Compute Canada. Furthermore, we also excluded the computation of the L-smoothness parameter for the coordinate descent method in our simulations.}

The simulation setups for penalized logistic models are similar to those above for penalized linear models, except that the active coefficients are set differently to account for the exponential scale inherent to the logistic regression. For the single-replicate visualization simulations, we let the $4$ nonzero signals be $0.5,-0.5,0.8,-0.8$. For the simulations with $100$ replications to compare the convergence rate and signal recovery performance, we simulate the $5$ blocks of the ``true'' signals from $N_{10}\left(0.5,1\right)$, $N_{10}\left(0.5,1\right)$,
$N_{10}\left(-0.5,1\right)$, $N_{10}\left(-0.5,1\right)$, $N_{10}\left(1,1\right)$, respectively. The SNR for logistic regression has the same definition as linear models, with Gaussian noise added to the generated continuous predictor $\mathbf{X}\boldsymbol{\beta}_{true}$. The binary outcomes are independent Bernoulli realizations, with probabilities being the logistic transforms of the continuous response.

\subsection{Simulation Results} \label{sec:sim-results}

\subsubsection{Penalized Linear Regression} \label{sec:linear-sim}

\begin{figure}[ht]
\centering{}\includegraphics[scale=0.5]{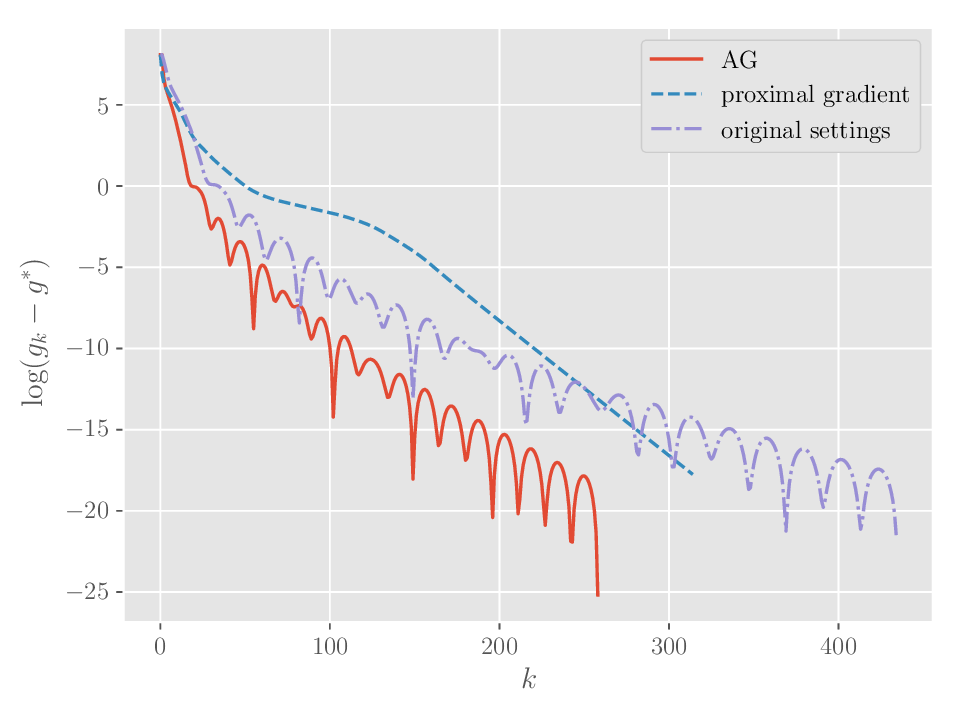}\includegraphics[scale=0.5]{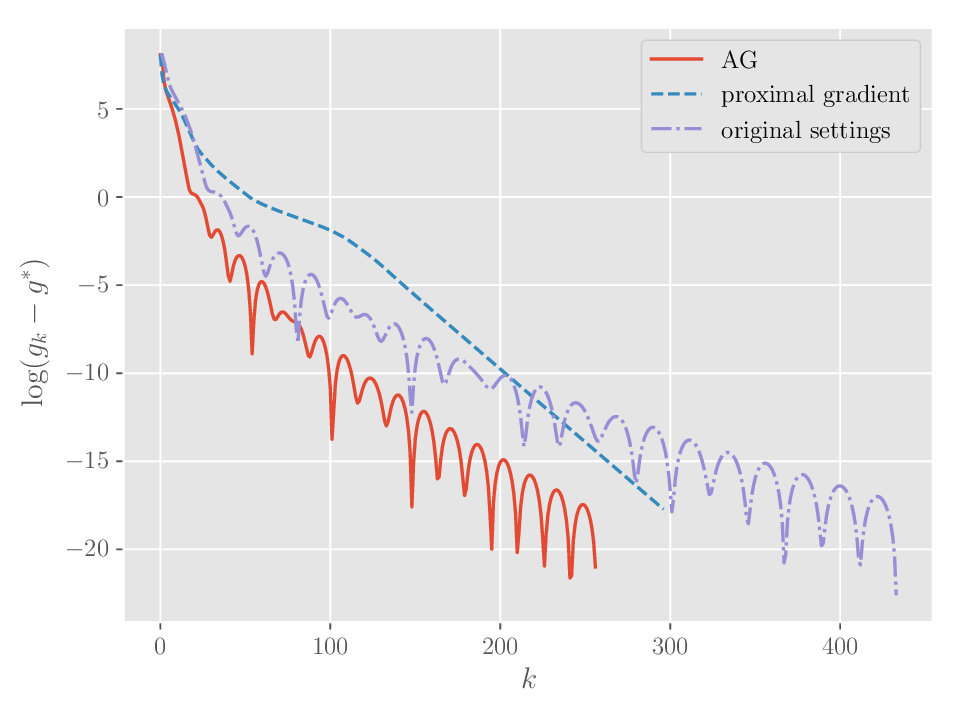}\caption{Convergence rate performance of first-order methods on SCAD (left) and MCP (right)
penalized linear model for a single simulation replicate. $k$ represents the number of iterations, $g_k$ represents the iterative objective function value, and $g^*$ represents the minimum found by the three methods considered. \label{fig:LM-iteration-plot}}
\end{figure}

Figure \ref{fig:LM-iteration-plot} shows the log differences of iterative objective values for a single replicate. This figure visualizes the accelerating effect of the AG method using our proposed hyperparameter settings. Median with the corresponding $95\%$ bootstrap CI of the number of iterations required for the iterative objective function values to make a fixed amount of descent for $100$ simulation replications are reported in Figures \ref{fig:sim-ag-lm-scad}, \ref{fig:sim-ag-lm-mcp} in Appendix \ref{sec:Simulations-LM}. The lack of bars in the reported barplots indicates that the median of $100$ replications breaks down; i.e., the corresponding proximal gradient algorithm fails to converge to the minimizer found by the three algorithms within $2000$ iterations. 
The AG method using our hyperparameter settings converges
much faster than proximal gradient and AG using the original hyperparameter settings
proposed by \citeauthor{Ghadimi2015} for both SCAD and MCP-penalized models discussed
here, as reflected in Figures \ref{fig:LM-iteration-plot}, \ref{fig:sim-ag-lm-scad}, \ref{fig:sim-ag-lm-mcp}. It can also be observed that momentum methods such
as AG are much less likely to be stuck at saddle points or local minimizers
than proximal gradient -- this property is consistent with previous findings~\parencite{Jin2017}. Since the proposed AG methods belong to the class
of momentum methods, the AG algorithms do not possess a descent property. As
suggested by a previous study~\parencite{Su2014}, oscillation will occur
at the end of the trajectory; the descent property will therefore
vanish. This is also reflected in Figures \ref{fig:LM-iteration-plot},
\ref{fig:logistic-iteration-plot} -- as the trajectory moves close
to the optimizer, the oscillation will start to occur for the AG methods. Among all the first-order methods, the AG method with our proposed hyperparameter settings tends to converge the fastest in all scenarios considered, as illustrated by Figures \ref{fig:sim-ag-lm-scad}, \ref{fig:sim-ag-lm-mcp} in Appendix \ref{sec:Simulations-LM}. The observed standard errors among $100$ simulation replications are rather small, suggesting that the halting time retains predictable for high-dimensional models, which agrees with the recent findings~\parencite{Paquette2020}. 

{ Figures \ref{fig:lm-time-scad}, \ref{fig:lm-time-mcp} report median with the corresponding $95\%$ bootstrap CI of the computing time (in seconds) required for the infinity norm of the two consecutive iterations $\left\Vert \boldsymbol{\beta}^{\left(k+1\right)}-\boldsymbol{\beta}^{\left(k\right)}\right\Vert _{\infty}$ to fall below $10^{-4}$ for $100$ simulation replications. It can be observed that the computing time for AG with suggested settings is much shorter than the computing time for coordinate descent.}

\begin{figure}[ht]
\begin{centering}
\includegraphics[scale=0.8]{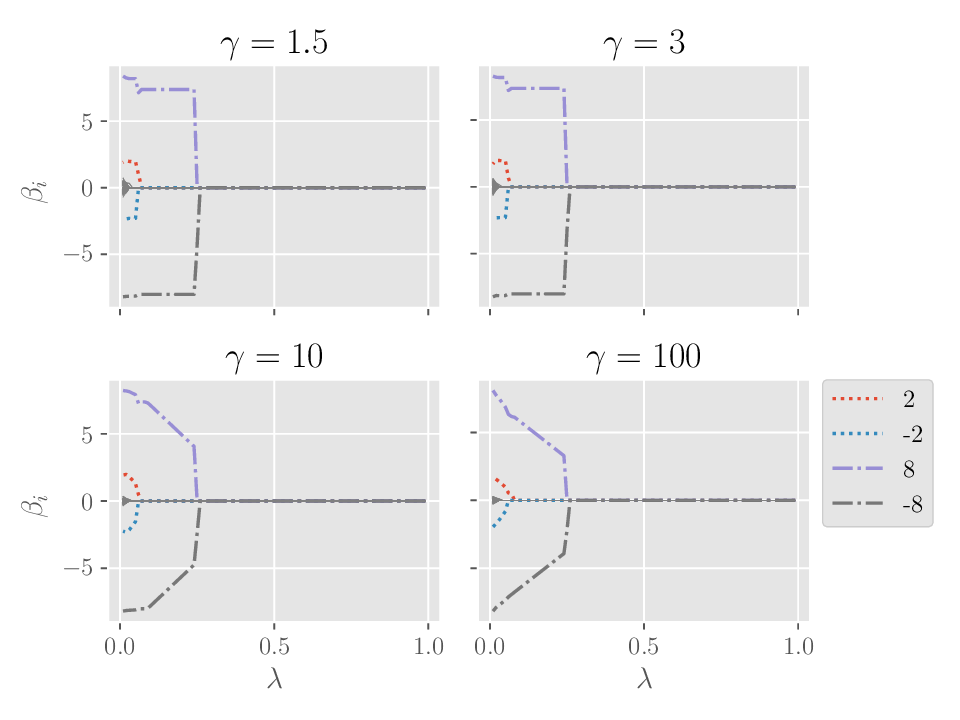}
\par\end{centering}
\caption{Solution paths obtained using the proposed AG method for MCP-penalized linear
model with different values of $\gamma$ for a single simulation replicate. The behaviors of the solution path match the expected from the MCP penalized problems. The solution path behaves similarly to hard-thresholding for a small $\gamma$. As $\gamma$ increases, the solution path will behave more similarly to soft-thresholding. \label{fig:linear-signal}}
\end{figure}

To visualize the signal recovery performance using our proposed method, Figure \ref{fig:linear-signal} plots the solution paths for
the MCP-penalized linear model with different values of $\gamma$. The grey lines in Figure \ref{fig:linear-signal} represent the recovered values for the noise variables. AG method performs very well when applied to signal recovery problems for nonconvex-penalized linear models. Figure \ref{fig:linear-signal} serves as an arbitrary instance that the recovered signals using our method exhibit the expected pattern with MCP -- as $\lambda$ decreases, the degree of penalization decreases, and more false-positive signals will be selected. The stable solution path for the recovered signals suggests that the algorithm does not converge to a point far away from the ``true'' coefficients.

\begin{figure}[ht]
\begin{centering}
\includegraphics[scale=0.8]{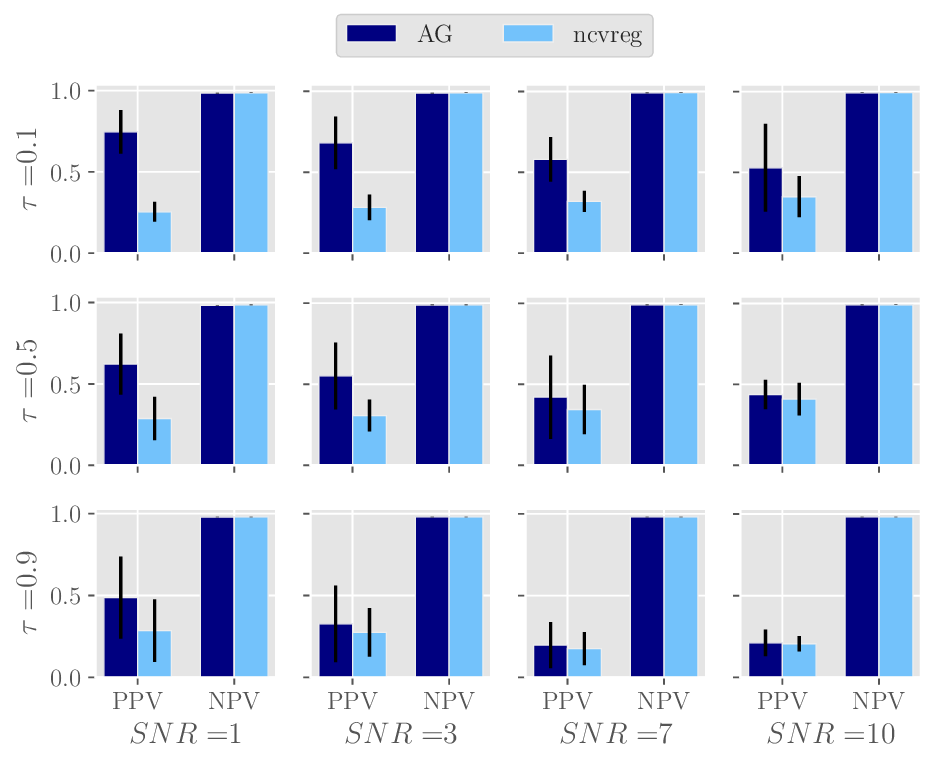}
\par\end{centering}
\caption{Sample means for Positive/Negative Predictive Values (PPV, NPV) of signal detection across different values of covariates correlation ($\tau$) and SNRs for AG with our proposed hyperparameter settings and {\tt ncvreg} on SCAD-penalized linear model over $100$ simulation replications. The error bars represent the standard errors.
\label{fig:LM-SCAD-pv}}
\end{figure}

To further illustrate the signal recovery performance, the means and standard errors for the scaled estimation error $\frac{\left\Vert \boldsymbol{\beta}_{\text{true}}-\hat{\boldsymbol{\beta}}\right\Vert _{2}^{2}}{\left\Vert \boldsymbol{\beta}_{\text{true}}\right\Vert _{2}^{2}}$, positive/negative predictive values (PPV, NPV), and active set cardinality across $100$ replications are reported in Tables \ref{tab:signal-lm-scad} and \ref{tab:signal-lm-mcp} in Appendix \ref{sec:Simulations-LM}. In what follows, $\mathcal{A}$ denotes the set of nonzero ``true'' coefficients and $\hat{\mathcal{A}}$ denotes the set of nonzero coefficients selected by the model. PPV and NPV use the following definitions: 
\begin{equation*}
\text{PPV}\coloneqq \frac{\lvert\mathcal{A}\cap\hat{\mathcal{A}}\rvert}{\lvert\hat{\mathcal{A}}\rvert},\ \text{NPV}\coloneqq \frac{\lvert\mathcal{A}^{C}\cap\hat{\mathcal{A}}^{C}\rvert}{\lvert\hat{\mathcal{A}}^{C}\rvert}.
\end{equation*}
Sample means and standard errors for PPV and NPV from Table \ref{tab:signal-lm-scad} are further visualized in Figure \ref{fig:LM-SCAD-pv}. When applied to sparse
learning problems, the signal recovery performance of our proposed
method often outperforms {\tt ncvreg}, the current state-of-the-art method~\parencite{Breheny2011}, particularly in terms of the positive predictive values (PPV). This can be observed from \figref{fig:LM-SCAD-pv} and Tables \ref{tab:signal-lm-scad}, \ref{tab:signal-lm-mcp} from Appendix \ref{sec:Simulations-LM}. This observation is especially evident when the signal-to-noise ratios are low. At the same time, $\ensuremath{\nicefrac{\left\Vert \boldsymbol{\beta}_{\text{true}}-\hat{\boldsymbol{\beta}}\right\Vert _{2}^{2}}{\left\Vert \boldsymbol{\beta}_{\text{true}}\right\Vert _{2}^{2}}}$ for both methods are close. As the SNR increases, the validation set becomes more similar to the training set, causing the chosen model to have a smaller $\lambda$. The model size will therefore increase, which will decrease the value of PPV.

\subsubsection{Penalized Logistic Regression} \label{sec:logistic-sim}

\begin{figure}[H]
\centering{}\includegraphics[scale=0.5]{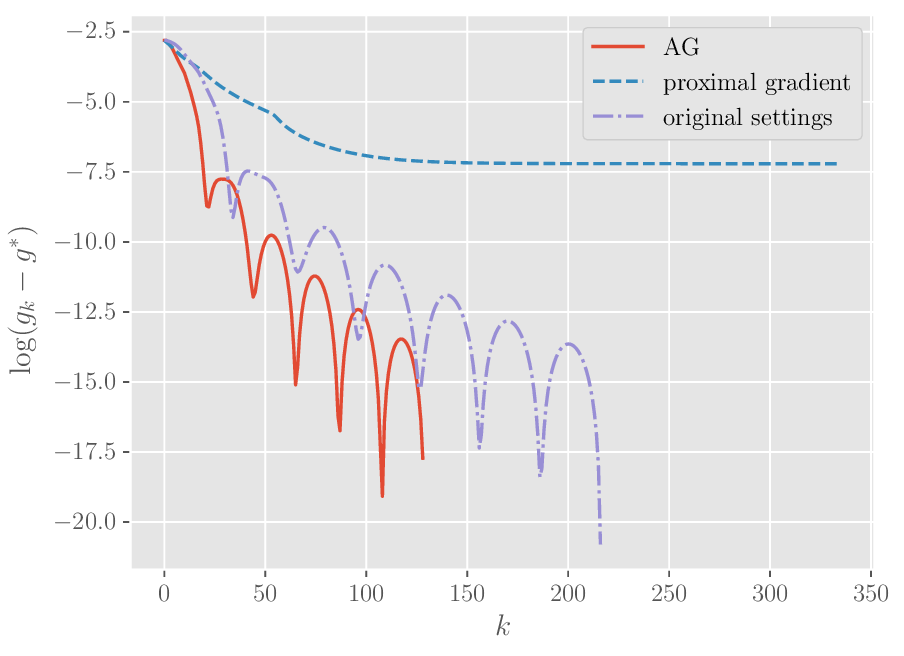}\includegraphics[scale=0.5]{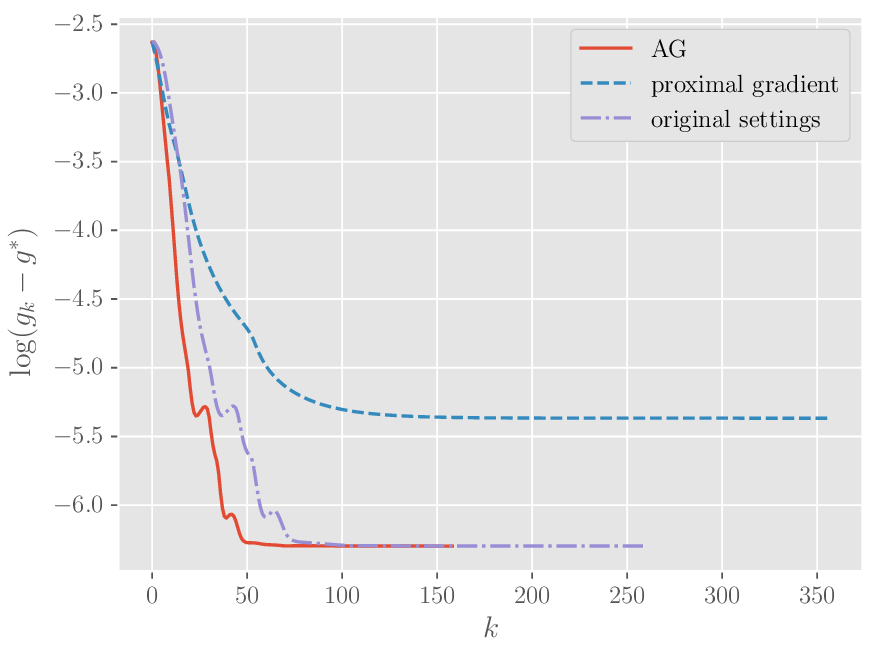}\caption{Convergence rate performance of first-order methods on SCAD (left) and MCP (right)
penalized logistic regression for a single simulation replicate. $k$ represents the number of iterations, $g_k$ represents the iterative objective function value, and $g^*$ represent the minimum found by the three methods considered. \label{fig:logistic-iteration-plot}}
\end{figure}

\begin{figure}[ht]
\begin{centering}
\includegraphics[scale=0.8]{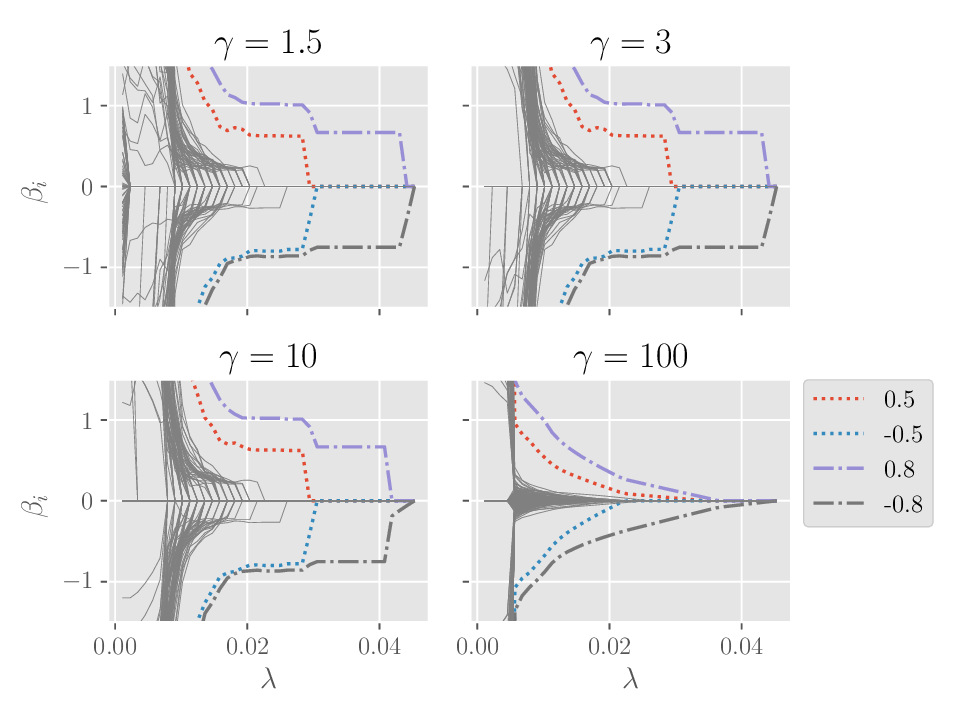}
\par\end{centering}
\caption{Solution paths obtained using the proposed AG method for MCP-penalized logistic
regression with different values of $\gamma$ for a single simulation replicate. The behaviors of the solution path match the expected from the MCP penalized problems. The solution path behaves similarly to hard-thresholding for a small $\gamma$. As $\gamma$ increases, the solution path will behave more similarly to soft-thresholding. \label{fig:logistic-signal}}
\end{figure}

\begin{figure}[ht]
\begin{centering}
\includegraphics[scale=0.8]{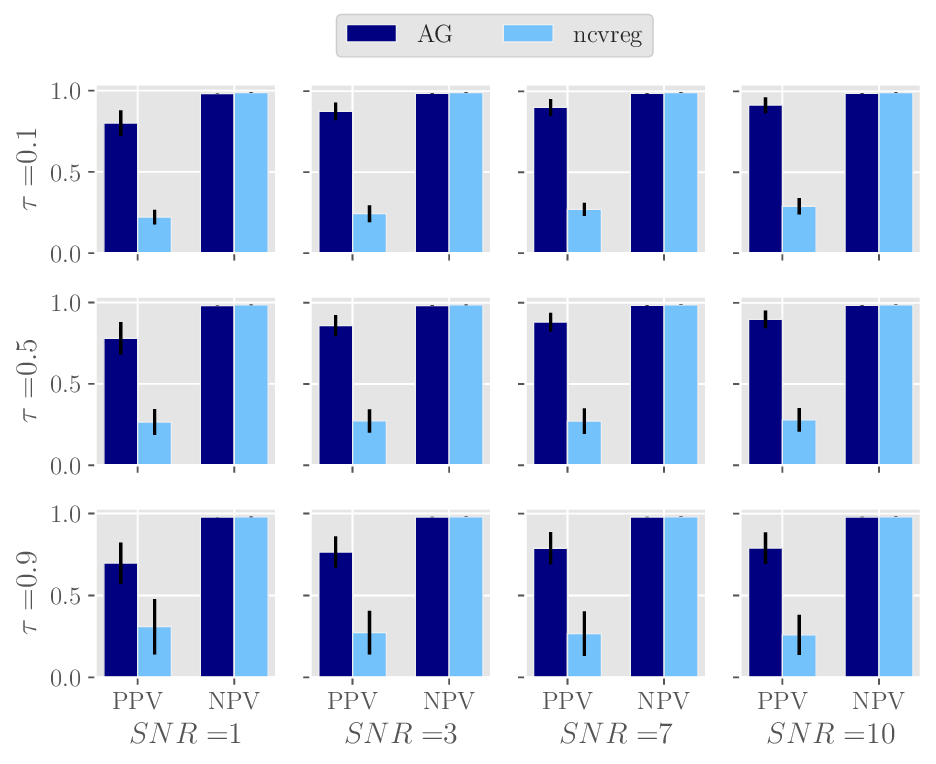}
\par\end{centering}
\caption{Sample means for Positive/Negative Predictive Values (PPV, NPV) of signal detection across different values of covariates correlation ($\tau$) and SNRs for AG with our proposed hyperparameter settings and {\tt ncvreg} on SCAD-penalized logistic model over $100$ simulation replications. The error bars represent the standard error.
\label{fig:logistic-SCAD-pv}}
\end{figure}

The simulation results reflected in Figures \ref{fig:logistic-iteration-plot}, \ref{fig:logistic-signal}, as well as Figures \ref{fig:sim-ag-logistic-scad}, \ref{fig:sim-ag-logistic-mcp} and Tables \ref{tab:signal-logistic-scad}, \ref{tab:signal-logistic-mcp} in Appendix \ref{sec:Simulations-logistic} suggest similar findings for penalized logistic models to our findings for penalized linear models as discussed in Section \ref{sec:linear-sim}. We further note that when applied to penalized logistic models, the coordinate descent method often fails to converge, resulting in overall poor performance in positive predictive values as reflected in Figure \ref{fig:logistic-SCAD-pv} and Tables \ref{tab:signal-logistic-scad}, \ref{tab:signal-logistic-mcp} in Appendix \ref{sec:Simulations-logistic}. When it does converge, the coordinate descent method does so at a very slow rate. In comparison, our proposed method has a convergence guarantee in theory and converges within a reasonable number of iterations in our simulation studies, as shown in Figures \ref{fig:sim-ag-lm-scad}, \ref{fig:sim-ag-lm-mcp} in Appendix \ref{sec:Simulations-logistic}. { In our computing time comparison, we used identical simulation setups and convergence standard for both the AG method and coordinate descent method, running both on a NVIDIA A100 GPU with CUDA compute capability of 8.0 from Compute Canada; the submitted simulation job finished well within $20$ minutes for both SCAD and MCP-penalized logistic models when using the AG method, but exceeded the 7-day computing time limit imposed on the Narval cluster when using the coordinate descent method.}

\section{Discussion}


We considered a recently developed generalization of Nesterov's accelerated gradient method for nonconvex optimization, and we have discussed its potential in sparse statistical learning with nonconvex penalties. An important issue concerning this algorithm is the selection of its sequences of hyperparameters. We present an explicit solution to this problem by minimizing the algorithm's complexity upper bound, hence accelerating convergence of the algorithm. 
Our simulation studies indicate that among first-order methods, the AG method using our proposed hyperparameter settings achieves a convergence rate considerably faster than other first-order methods such as the AG method using the original proposed hyperparameter settings or proximal gradient. Our simulations also show that signal recovery using our proposed method generally outperforms {\tt ncvreg}, the current state-of-the-art method. This performance gain is much more pronounced for penalized linear models when the signal-to-noise ratios are low. For penalized logistic regression, the performance gain observed is consistent across various covariates correlation and signal-to-noise ratio settings. Compared to coordinate-wise minimization methods, our proposed method is less challenged by low signal-to-noise ratios and is feasible to implement in parallel. { Given today's computing facilities, parallel computing is particularly meaningful for large datasets~\parencite{Parnell2020}. We also show this gain in parallel computing performance by comparing computing time on a GPU.} Furthermore, our proposed method has weaker convergence conditions and can be applied to a class of problems that do not have an explicit solution to the coordinate-wise objective function. For example, linear mixed models for grouped or longitudinal data involve the inverse of a large covariance matrix. Decomposition of this covariance matrix is necessary to apply the coordinate descent method. However, such decomposition can be computationally costly and numerically unstable~\parencite{Quarteroni2000}. On the other hand, matrix decomposition is not needed for first-order methods, as numerically stable yet computationally efficient approaches such as conjugate gradient can be adapted when applying our proposed method. The proposed nonconvex AG method can be applied to a wide range of statistical learning problems, opening various future research opportunities in statistical machine learning and statistical genetics. 

\section{Disclaimer}
All codes to reproduce the simulation results of this paper and outputs from Calcul Quebec/Compute Canada can be found on the following GitHub repository:

\href{https://github.com/Kaiyangshi-Ito/nonconvexAG}{https://github.com/Kaiyangshi-Ito/nonconvexAG}

\newpage

\appendix

\section{Proofs} \label{sec:proof}
We first establish the following Lemma needed for the proof of Theorem \ref{thm:convex}.
\subsection{Proof of Theorem \ref{thm:convex} }
\label{sec:proof-thm-1}

The following lemma is needed in the proof of Theorem 1. 
\begin{lem}
\label{lem:convergence-cond-meaning} Assume that $\forall k=1,2,\dots,N$,
the convergence conditions \eqref{eq:convcond1} and \eqref{eq:convcond2}
hold, then we have the following recursive relation: 
\begin{equation}
\alpha_{k+1}\leq\frac{1}{1+\frac{\delta_{k}/\delta_{k+1}}{\alpha_{k}}}\label{eq:bifurcation}.
\end{equation}
\end{lem}
\begin{proof}
The convergence conditions \eqref{eq:convcond1} and \eqref{eq:convcond2}
gives that $\forall k=1,2,\dots,N-1$, 
\[
\alpha_{k+1}\delta_{k+1}\leq\omega_{k+1}\Leftrightarrow\alpha_{k+1}\leq\frac{\omega_{k+1}}{\delta_{k+1}},\text{ and}
\]
\[
\frac{\alpha_{k}}{\delta_{k}\Gamma_{k}}\geq\frac{\alpha_{k+1}}{\delta_{k+1}\Gamma_{k+1}}\Leftrightarrow\frac{\alpha_{k}}{\delta_{k}}\geq\frac{\alpha_{k+1}}{\delta_{k+1}\left(1-\alpha_{k+1}\right)}\Leftrightarrow\alpha_{k+1}\leq\frac{\alpha_{k}\delta_{k+1}}{\alpha_{k}\delta_{k+1}+\delta_{k}}.
\]
Following above two inequalities, we have that 
\begin{equation}
\alpha_{k+1}\leq\min\left\{ \frac{\omega_{k+1}}{\delta_{k+1}},\frac{\alpha_{k}\delta_{k+1}}{\alpha_{k}\delta_{k+1}+\delta_{k}}\right\} \label{eq:alpha-upper-bound}.
\end{equation}
We observe that in \eqref{eq:alpha-upper-bound}, $\frac{\omega_{k+1}}{\delta_{k+1}}$
is monotonically decreasing with respect to $\delta_{k+1}$ on $\mathbb{R}_{+}$;
while $\frac{\alpha_{k}\delta_{k+1}}{\alpha_{k}\delta_{k+1}+\delta_{k}}$
is monotonically increasing with respect to $\delta_{k+1}$ on $\mathbb{R}_{+}$.
This suggests: 
\begin{equation}
\arg\max_{\delta_{k+1}>0}\left(\min\left\{ \frac{\omega_{k+1}}{\delta_{k+1}},\frac{\alpha_{k}\delta_{k+1}}{\alpha_{k}\delta_{k+1}+\delta_{k}}\right\} \right)=\left\{ \frac{\omega_{k+1}+\sqrt{\omega_{k+1}^{2}+\frac{4\omega_{k+1}\delta_{k}}{\alpha_{k}}}}{2}\right\} .\label{eq:get-lambda}
\end{equation}
That is, the inequality constraints conditions \eqref{eq:convcond1}
and \eqref{eq:convcond2} for convergence are merely a lower bound
on the\emph{ vanishing rate} of $\left\{ \alpha_{k}\right\} $. Therefore
it follows from \eqref{eq:convcond1} and the (necessary) optimality
condition for \eqref{eq:get-lambda} that 
\begin{equation}
\alpha_{k+1}\leq\frac{2\omega_{k+1}}{\omega_{k+1}+\sqrt{\omega_{k+1}^{2}+\frac{4\omega_{k+1}\delta_{k}}{\alpha_{k}}}}\leq\frac{2}{1+\sqrt{1+\frac{4\delta_{k}}{\alpha_{k}\omega_{k+1}}}}=\frac{2}{1+\sqrt{1+\frac{4\delta_{k}/\delta_{k+1}}{\alpha_{k}\alpha_{k+1}}}}.\label{eq:alpha-bound}
\end{equation}
By simplifying \eqref{eq:bifurcation}, we have: 
\[
\alpha_{k+1}\leq\frac{1}{1+\frac{\delta_{k}/\delta_{k+1}}{\alpha_{k}}}.
\]
\end{proof}

We now proceed with the proof of Theorem 1. 
\begin{proof}
The complexity upper bound \eqref{eq:complexity-bound} under the
given conditions can be simplified as: 
\begin{align}
 & \left[\sum_{k=1}^{N}\Gamma_{k}^{-1}\omega_{k}\left(1-L_{\Psi}\omega_{k}\right)\right]^{-1}\left[\frac{\left\Vert x_{0}-x^{*}\right\Vert ^{2}}{\delta_{1}}+\frac{2L_{f}}{\Gamma_{N}}\left(\left\Vert x^{*}\right\Vert ^{2}+M^{2}\right)\right]\nonumber \\
= & \left[\sum_{k=1}^{N}\Gamma_{k}^{-1}\omega_{k}\left(1-L_{\Psi}\omega_{k}\right)\right]^{-1}\cdot\frac{\left\Vert x_{0}-x^{*}\right\Vert ^{2}}{\delta_{1}}\nonumber \\
= & \frac{1}{\omega\left(1-L_{\Psi}\omega\right)}\left(\sum_{k=1}^{N}\Gamma_{k}^{-1}\right)^{-1}\cdot\frac{\left\Vert x_{0}-x^{*}\right\Vert ^{2}}{\omega}\nonumber \\
= & \left(\sum_{k=1}^{N}\Gamma_{k}^{-1}\right)^{-1}\cdot\frac{\left\Vert x_{0}-x^{*}\right\Vert ^{2}}{\omega^{2}\left(1-L_{\Psi}\omega\right)}\label{eq:convex-complex-bound}.
\end{align}
Observe that $\left(\sum_{k=1}^{N}\Gamma_{k}^{-1}\right)^{-1}$ is
monotonically decreasing with respect to $\alpha_{k}$ for all $k=1,2,\dots,N$.
This property implies that \eqref{eq:convex-complex-bound} is minimized when
$\alpha_{k}$ attains its greatest value for $k=1,2,\dots,N$. 

Condition $\delta_{1}=\omega_{k}=\omega$ gives that 
\[
\omega_{1}=\delta_{1}=\alpha_{1}\delta_{1}.
\]
Since the upper bound for $\alpha_{k+1}$ presented in \eqref{eq:bifurcation}
is monotonically increasing with respect to $\alpha_{k}$, it then
follows inductively from the (necessary) optimality condition of \eqref{eq:alpha-upper-bound}
that 
\[
\alpha_{k+1}\leq\frac{1}{1+\frac{\delta_{k}/\delta_{k+1}}{\alpha_{k}}}=\frac{1}{1+\frac{\alpha_{k+1}}{\alpha_{k}^{2}}},
\]
which simplifies to 
\[
\alpha_{k+1}\leq\frac{2}{1+\sqrt{1+\frac{4}{\alpha_{k}^{2}}}}.
\]
While $\omega^{2}\left(1-L_{\Psi}\omega\right)$ should be maximized
to minimize the value of \eqref{eq:convex-complex-bound}, which implies
the minimizer for $\omega$ is 
\[
\bar{\omega}=\frac{2}{3L_{\Psi}}.
\]
And $\bar{\lambda}_{k+1}=\frac{\bar{\omega}}{\bar{\alpha}_{k+1}}$
follows directly form the necessary optimality condition for \eqref{eq:alpha-upper-bound}. It is trivial to check that $\left(\left\{ \bar{\alpha}_{k}\right\} ,\left\{ \bar{\delta}_{k}\right\} ,\bar{\omega}\right)$ is feasible under given constraints \eqref{eq:convcond1} and \eqref{eq:convcond2}.
\end{proof}

\subsection{Proof of Theorem \ref{thm:1overk}}
\label{sec:proof-thm-2}
\begin{proof}
Consider arbitrary $k=2,\dots,N$, then $\alpha_{k}\in\left(0,1\right)$
by definition. In the convergence conditions \eqref{eq:convcond1}
and \eqref{eq:convcond2}, this gives us that 
\[
\frac{\alpha_{k+1}}{\alpha_{k}}\leq\frac{2}{\alpha_{k}+\sqrt{\alpha_{k}^{2}+4}}\in\left(\frac{\sqrt{5}-1}{2},1\right).
\]
Thus, $\left\{ \alpha_{k}\right\} $ is a bounded monotonically decreasing
sequence, and $\alpha_{2}\leq\frac{2}{1+\sqrt{1+\frac{4}{1^{2}}}}=\frac{\sqrt{5}-1}{2}$
further implies that $\forall k\geq2,\ \alpha_{k}\in(0,\frac{\sqrt{5}-1}{2}]$. 

For all $k\geq2$, $\alpha_{k}\in\left(0,1\right)$ implies that $1-\alpha_{k}\in\left(0,1\right)$.
Therefore, $\Gamma_{k}^{-1}=\frac{1}{\left(1-\alpha_{2}\right)\left(1-\alpha_{3}\right)\cdots\left(1-\alpha_{k}\right)}$
is monotonically increasing with respect to $k$. Thus, $\sum_{k=1}^{N}\Gamma_{k}^{-1}=O\left(N\right)$,
which implies that $\left(\sum_{k=1}^{N}\Gamma_{k}^{-1}\right)^{-1}\cdot C_{1}=O\left(1/N\right)$. 

Observe that 
\begin{align}
 & 0<\left(\Gamma_{N}\sum_{k=1}^{N}\frac{1}{\Gamma_{k}}\right)^{-1}=\frac{1}{N\cdot\Gamma_{N}}\cdot\frac{N}{\sum_{k=1}^{N}\frac{1}{\Gamma_{k}}}\nonumber \\
 & \leq\frac{1}{N\cdot\Gamma_{N}}\cdot\left(\prod_{k=1}^{N}\Gamma_{k}\right)^{\frac{1}{N}}=\frac{1}{N}\cdot\left(\prod_{k=1}^{N}\frac{\Gamma_{k}}{\Gamma_{N}}\right)^{\frac{1}{N}}\label{eq:HM-GM}\\
 & =\frac{1}{N}\cdot\left(\prod_{k=1}^{N}\frac{\Gamma_{N}}{\Gamma_{k}}\right)^{-\frac{1}{N}}=\frac{1}{N}\cdot\left(\prod_{k=2}^{N}\left(1-\alpha_{k}\right)^{k}\right)^{-\frac{1}{N}}\nonumber \\
 & =\frac{1}{N}\cdot\prod_{k=2}^{N}\left(1-\alpha_{k}\right)^{-\frac{k}{N}},\nonumber 
\end{align}
where the inequality in \eqref{eq:HM-GM} follows from the harmonic
mean-geometric mean inequality. 

Consider arbitrary $N\in\mathbb{N}$, now we are to prove that $\forall k=1,2,\dots,N,\ \alpha_{k}\leq\frac{2}{k+1}$.
By definition, $\alpha_{1}=1\leq1$. Assume that $\alpha_{k}\leq\frac{2}{k+1}$,
then by the convergence conditions, 
\begin{align*}
\alpha_{k+1}\leq & \frac{2}{1+\sqrt{1+\frac{4}{\alpha_{k}^{2}}}}\\
\leq & \frac{2}{1+\sqrt{1+4/\left(\frac{2}{k+1}\right)^{2}}}\\
= & \frac{2}{1+\sqrt{2+2k+k^{2}}}\\
< & \frac{2}{k+2}.
\end{align*}
Thus, by mathematical induction, $\forall k=1,2,\dots,N,\ \alpha_{k}\leq\frac{2}{k+1}$.
Hence, $\sum_{k=1}^{N}\frac{k}{N}\alpha_{k}<\sum_{k=1}^{N}\frac{k}{N}\cdot\frac{2}{k}=\sum_{k=1}^{N}\frac{2}{N}=2<\infty$
as $N\rightarrow\infty$. 

Furthermore, we have that $\forall x\in(0,\frac{\sqrt{5}-1}{2}],\ -\log\left(1-x\right)<x$.
Combined with the fact that $\forall k\geq2,\ \alpha_{k}\in(0,\frac{\sqrt{5}-1}{2}]$,
we have that $\forall k\geq2,\ -\log\left(1-\alpha_{k}\right)<\alpha_{k}$.
Thus, 
\[
\log\left(\prod_{k=2}^{N}\left(1-\alpha_{k}\right)^{-\frac{k}{N}}\right)=-\sum_{k=2}^{N}\frac{k}{N}\log\left(1-\alpha_{k}\right)<\sum_{k=2}^{N}\frac{k}{N}\alpha_{k}\leq2<\infty.
\]

Therefore, $\prod_{k=2}^{N}\left(1-\alpha_{k}\right)^{-\frac{k}{N}}$
is also upper bounded as $N\rightarrow\infty$, which implies that
\[
\left(\sum_{k=1}^{N}\frac{\Gamma_{N}}{\Gamma_{k}}\right)^{-1}\leq\frac{1}{N}\cdot\prod_{k=2}^{N}\left(1-\alpha_{k}\right)^{-\frac{k}{N}}=O\left(1/N\right).
\]
Hence, $\left(\sum_{k=1}^{N}\frac{\Gamma_{N}}{\Gamma_{k}}\right)^{-1}\cdot C_{2}=O\left(1/N\right)$.
Therefore, $\left(\sum_{k=1}^{N}\Gamma_{k}^{-1}\right)^{-1}\cdot C_{1}+\left(\sum_{k=1}^{N}\frac{\Gamma_{N}}{\Gamma_{k}}\right)^{-1}\cdot C_{2}=O\left(1/N\right)$. 
\end{proof}

\subsection{Proof of Theorem \ref{thm:alpha-k-vanishing}}
\label{sec:proof-thm-3}
\begin{proof}
$\bar{\alpha}_{k}\leq\frac{2}{k+1}$ for $k=1,2,\dots,N$ has already been proved
in the proof of Theorem \ref{thm:1overk}. For the left inequality,
note that $\bar{\alpha}_{1}=1\geq\frac{2}{2+a}$ for $a>0$; for $k\geq2$,
we are to prove a stronger inequality: 
\begin{equation}
\bar{\alpha}_{k}\geq\frac{2}{\sqrt{\left(1+a\cdot k^{-b}\right)k\left[\left(1+a\cdot k^{-b}\right)k+2\right]}}\label{eq:alpha-close-lower}.
\end{equation}
For $k=2$, condition \eqref{eq:k-vanishing-cond} implies that
\begin{equation}
\label{eq:k-vanising-form2}
a\cdot2^{-b}\geq\frac{1}{\left(1-b\right)\left(4-b\right)}>\frac{1}{4}>\sqrt{5}-2 \text{ for }0<b<1,
\end{equation}
which suggests $\bar{\alpha}_{2}=\frac{2}{1+\sqrt{5}}\geq\frac{2}{\sqrt{\left(1+a\cdot2^{-b}\right)\cdot2\left[\left(1+a\cdot2^{-b}\right)\cdot2+2\right]}}$
by simple algebra. Assume \eqref{eq:alpha-close-lower} holds for
$k=t$, then 
\begin{align}
\bar{\alpha}_{t+1}= & \frac{2}{1+\sqrt{1+\frac{4}{\bar{\alpha}_{t}^{2}}}}\nonumber \\
\geq & \frac{2}{1+\sqrt{1+4/\left(2/\sqrt{\left(1+a\cdot t^{-b}\right)t\left[\left(1+a\cdot t^{-b}\right)t+2\right]}\right)^{2}}}\nonumber \\
= & \frac{2}{1+\sqrt{1+\left(1+a\cdot t^{-b}\right)t\left[\left(1+a\cdot t^{-b}\right)t+2\right]}}\nonumber \\
= & \frac{2}{\left(1+a\cdot t^{-b}\right)t+2}\nonumber \\
\geq & \frac{2}{\sqrt{\left(1+a\cdot\left(t+1\right)^{-b}\right)\left(t+1\right)\left[\left(1+a\cdot\left(t+1\right)^{-b}\right)\left(t+1\right)+2\right]}}\label{eq:many-terms};
\end{align}
and \eqref{eq:many-terms} follows from 
\begin{align}
 & \left(1+a\cdot\left(t+1\right)^{-b}\right)\left(t+1\right)\left[\left(1+a\cdot\left(t+1\right)^{-b}\right)\left(t+1\right)+2\right]-\left[\left(1+a\cdot t^{-b}\right)t+2\right]^{2}\nonumber \\
= & a^{2}\left[\left(t+1\right)^{2-2b}-t^{2-2b}\right]+2at\left[\left(t+1\right)^{1-b}-t^{1-b}\right]+4a\left[\left(t+1\right)^{1-b}-t^{1-b}\right]-1\nonumber \\
\geq & 2at\left[\left(t+1\right)^{1-b}-t^{1-b}\right]-1\nonumber \\
= & 2at^{2-b}\left[\left(1+\frac{1}{t}\right)^{1-b}-1\right]-1\nonumber \\
\geq & 2at^{2-b}\left[1+\left(1+b\right)t^{-1}-\frac{1}{2}b\left(1-b\right)t^{-2}-1\right]-1\label{eq:binom-approx}\\
= & 2a\left(1-b\right)t^{1-b}-ab\left(1-b\right)t^{-b}-1\geq0\label{eq:cond-2-works}.
\end{align}
\eqref{eq:binom-approx} follows from binomial approximation
inequality; $a>0$ and $0<b<1$ suggest that $2a\left(1-b\right)k^{1-b}-ab\left(1-b\right)k^{-b}-1$
is monotonically increasing with respect to $k$ for $k>0$, condition
\eqref{eq:k-vanishing-cond} therefore implies that $2a\left(1-b\right)k^{1-b}-ab\left(1-b\right)k^{-b}-1\geq0$
for all $k\geq2$, which is \eqref{eq:cond-2-works}. 

And proof of the left inequality for $k\geq2$ proceeds as the following:
\begin{align*}
\bar{\alpha}_{k}\geq & \frac{2}{\sqrt{\left(1+a\cdot k^{-b}\right)k\left[\left(1+a\cdot k^{-b}\right)k+2\right]}}\\
> & \frac{2}{\sqrt{\left(1+a\cdot k^{-b}\right)k\left[\left(1+a\cdot k^{-b}\right)k+2\right]+1}}\\
= & \frac{2}{\left(1+a\cdot k^{-b}\right)k+1}.
\end{align*}
\end{proof}

\subsection{Proof of Corollary \ref{cor:alpha-vanishing}}
\label{sec:proof-cor-1}
\begin{proof}
Observe that the lower bound of \eqref{eq:damping-bounds} is monotonically decreasing with respect to $a$ under given conditions. Constraint \eqref{eq:k-vanishing-cond} implies \eqref{eq:k-vanising-form2}, which further suggests that $$a\geq\frac{2^{b}}{\left(1-b\right)\left(4-b\right)}>0\text{ for }0<b<1;$$
i.e., $\bar{a}_k=\frac{\left(2/k\right)^{\bar{b}_k}}{\left(1-\bar{b}_k\right)\left(4-\bar{b}_k\right)}$. Thus, maximizing the lower bound of \eqref{eq:damping-bounds} is equivalent to minimize the convex function $\log\frac{\left(2/k\right)^{b}}{\left(1-b\right)\left(4-b\right)}$ with respect to $b$ over a open set $\left(0,1\right)$. First-order sufficient optimality condition gives the unique optimizer 
\[
\bar{b}_k=\frac{2+5\left(\log\frac{2}{k}\right)+\sqrt{9\left(\log\frac{2}{k}\right)^{2}+4}}{2\left(\log\frac{2}{k}\right)}\in\left(0,1\right)
\]
for $k\geq 8$. 
Simple algebra shows that $\lim_{k\rightarrow\infty}\frac{\bar{a}_kk^{1-\bar{b}_k}}{\log k}=\frac{2}{3}e$. Thus, the lower bound in Theorem \ref{thm:alpha-k-vanishing} becomes $\frac{k+1}{2}-\bar{\alpha}_{k}^{-1}=O\left(\log k\right)$.
\end{proof}

\section{Further Simulations}
\subsection{Penalized Linear Model \label{sec:Simulations-LM} }

In \figref{fig:sim-ag-lm-scad} and \ref{fig:sim-ag-lm-mcp},
the red bar represents AG using our proposed hyperparameter settings,
blue bar represents proximal gradient, and the purple bar represents AG using
the original hyperparameter settings~\parencite{Ghadimi2015}. It is evident
that for penalized linear models, AG using our hyperparameter settings
outperforms proximal gradient or AG using the original proposed hyperparameter settings considerably.

\begin{figure}[H]
\begin{centering}
\includegraphics[scale=1]{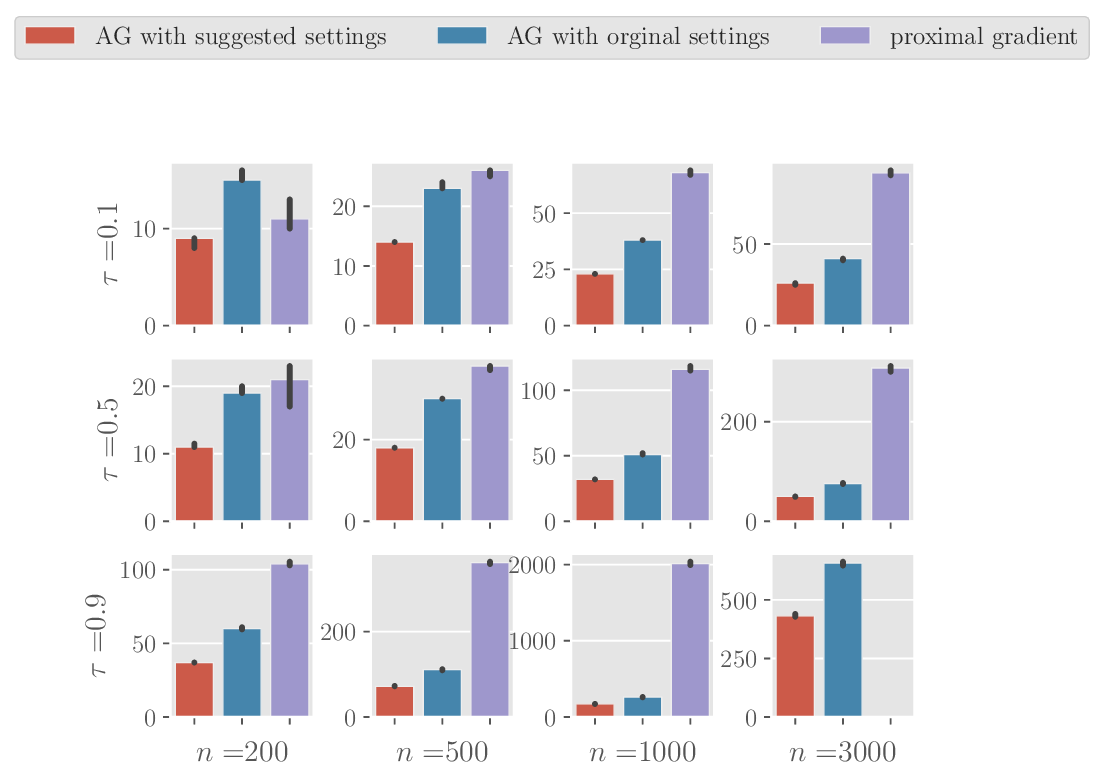}
\par\end{centering}
\caption{Median for the number of iterations required for the iterative objective value to reach $g^{*}+e^{3}$
on SCAD-penalized linear model for AG with our proposed hyperparameter
settings, AG with original settings, and proximal gradient over $100$ simulation replications, across varying covariates correlation ($\tau$) and $q/n$ values. The error bars represent the $95\%$ CIs from $1000$ bootstrap replications, $g^*$ represents the minimum per iterate found by the three methods considered.
\label{fig:sim-ag-lm-scad}}
\end{figure}

\begin{figure}[H]
\begin{centering}
\includegraphics[scale=1]{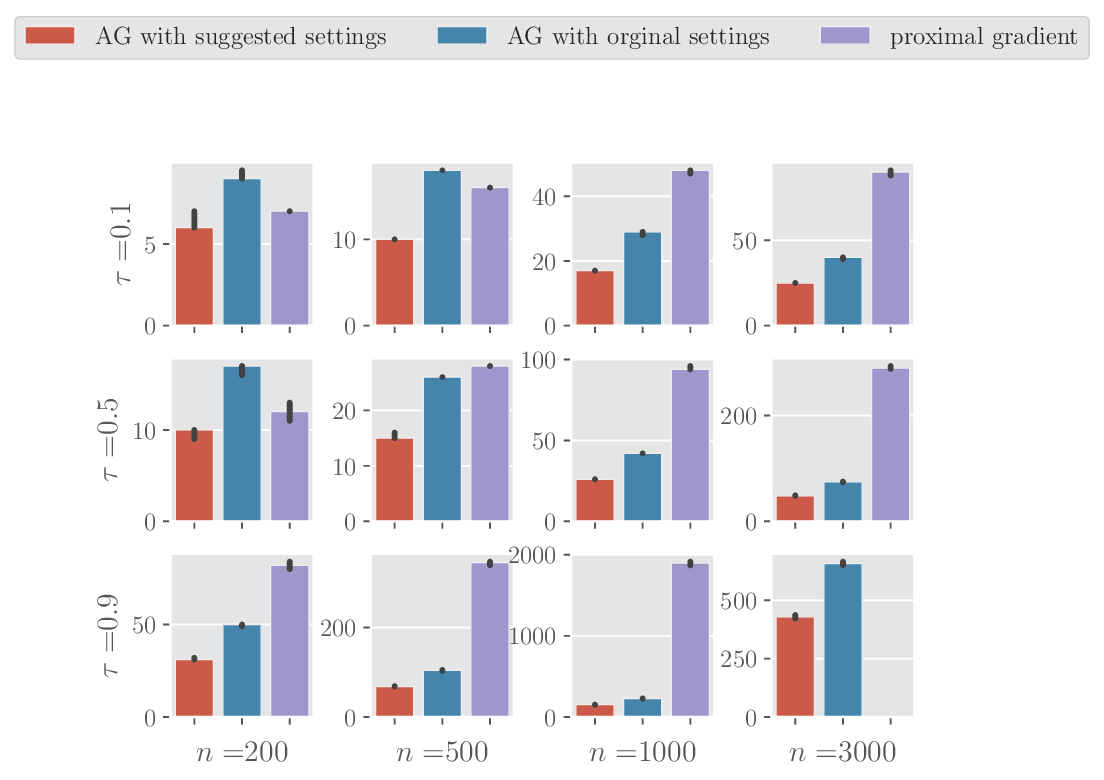}
\par\end{centering}
\caption{Median for the number of iterations required for iterative objective values to reach $g^{*}+e^{3}$
on MCP-penalized linear model for AG with our proposed hyperparameter settings, AG with original settings, and proximal gradient over $100$ simulation replications, across varying covariates correlation ($\tau$) and $q/n$ values. The error bars represent the $95\%$ CIs from $1000$ bootstrap replications, $g^*$ represents the minimum per iterate found by the three methods considered. \label{fig:sim-ag-lm-mcp}}
\end{figure}

\newpage
{ In \figref{fig:lm-time-scad} and \ref{fig:lm-time-mcp},
the red bar represents AG using our proposed hyperparameter settings,
blue bar represents proximal gradient, and the purple bar represents coordinate descent. It is evident
that for penalized linear models, AG using our hyperparameter settings
outperforms coordinate descent significantly in terms of computing time.}

\begin{figure}[H]
\begin{centering}
\includegraphics[scale=1]{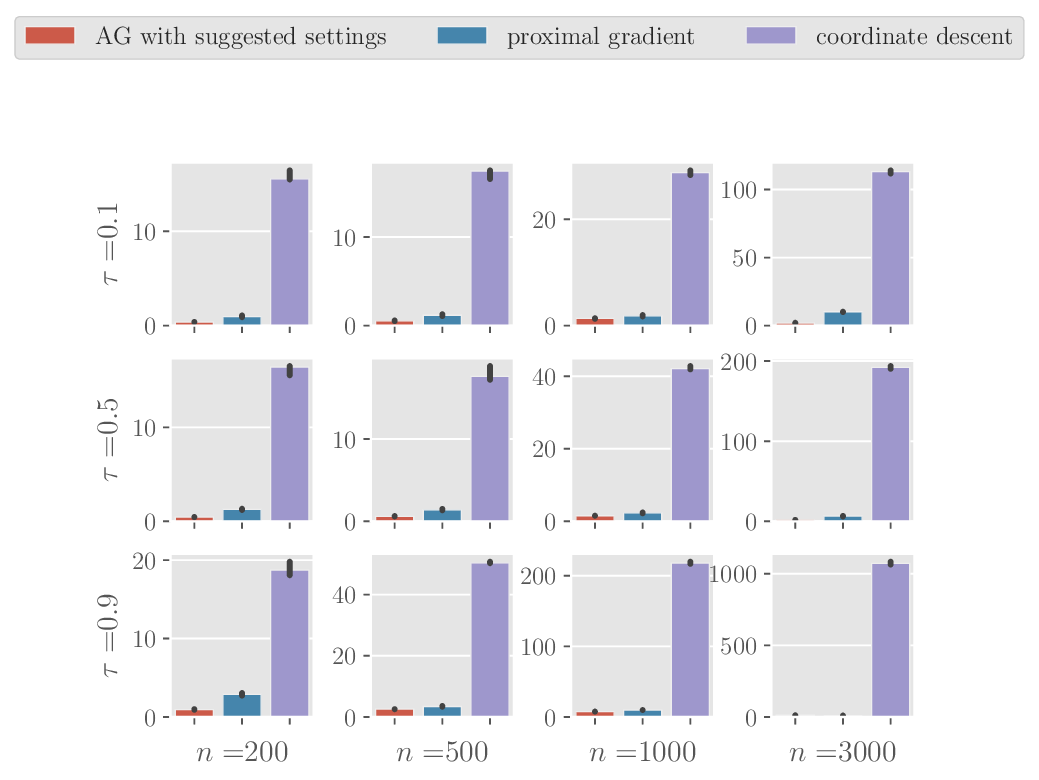}
\par\end{centering}
\caption{ Median for the computing time (in seconds) required for $\left\Vert \boldsymbol{\beta}^{\left(k+1\right)}-\boldsymbol{\beta}^{\left(k\right)}\right\Vert _{\infty}$ to fall below $10^{-4}$ on SCAD-penalized linear model for AG with our proposed hyperparameter
settings, proximal gradient, and coordinate descent over $100$ simulation replications, across varying covariates correlation ($\tau$) and $q/n$ values. The error bars represent the $95\%$ CIs from $1000$ bootstrap replications, $g^*$ represents the minimum per iterate found by the three methods considered.
\label{fig:lm-time-scad}}
\end{figure}

\begin{figure}[H]
\begin{centering}
\includegraphics[scale=1]{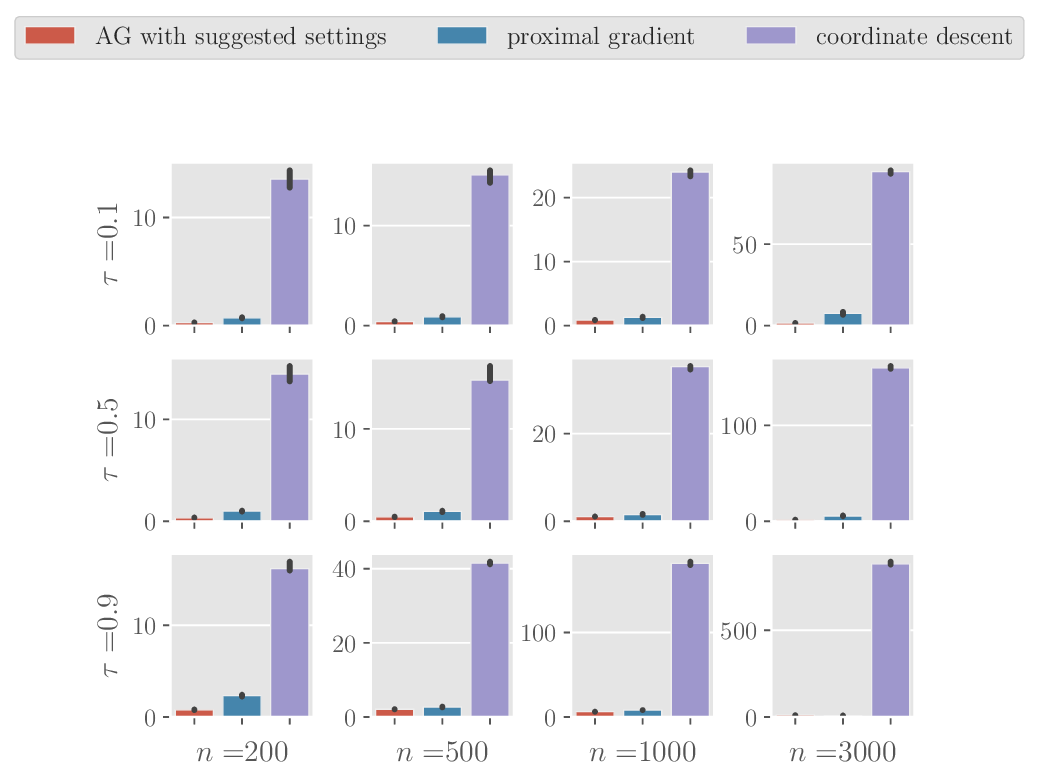}
\par\end{centering}
\caption{ Median for the computing time (in seconds) required for $\left\Vert \boldsymbol{\beta}^{\left(k+1\right)}-\boldsymbol{\beta}^{\left(k\right)}\right\Vert _{\infty}$ to fall below $10^{-4}$ on MCP-penalized linear model for AG with our proposed hyperparameter
settings, proximal gradient, and coordinate descent over $100$ simulation replications, across varying covariates correlation ($\tau$) and $q/n$ values. The error bars represent the $95\%$ CIs from $1000$ bootstrap replications, $g^*$ represents the minimum per iterate found by the three methods considered.
\label{fig:lm-time-mcp}}
\end{figure}

\newpage

\begin{table}[H]
\begin{center}

\caption{Signal recovery performance (sample mean and standard error of $\nicefrac{\left\Vert \boldsymbol{\beta}_{\text{true}}-\hat{\boldsymbol{\beta}}\right\Vert _{2}^{2}}{\left\Vert \boldsymbol{\beta}_{\text{true}}\right\Vert _{2}^{2}}$, Positive/Negative Predictive Values (PPV, NPV) for signal detection, and active set cardinality $\lvert\hat{\mathcal{A}}\rvert$) for {\tt ncvreg} and AG with our proposed hyperparameter settings on SCAD-penalized linear model over $100$ simulation replications, across varying values of SNRs and covariates correlations ($\tau$). \label{tab:signal-lm-scad}}
\begin{tabular}{c|ccc}
\hline 
\multicolumn{1}{c|}{$\nicefrac{\left\Vert \boldsymbol{\beta}_{\text{true}}-\hat{\boldsymbol{\beta}}\right\Vert _{2}^{2}}{\left\Vert \boldsymbol{\beta}_{\text{true}}\right\Vert _{2}^{2}}$} & $\tau=0.1$ & $0.5$ & $0.9$\\
\hline 
$\text{SNR}=1$, AG &  $0.128 (0.021)$ &  $0.521 (0.114)$ &  $2.839 (0.497)$ \\
$\text{SNR}=1$, \texttt{ncvreg} &   $0.131 (0.02)$ &  $0.485 (0.102)$ &  $2.929 (0.525)$ \\
$\text{SNR}=3$, AG &   $0.05 (0.009)$ &  $0.156 (0.035)$ &  $2.075 (0.339)$ \\
$\text{SNR}=3$, \texttt{ncvreg} &  $0.052 (0.009)$ &  $0.156 (0.028)$ &  $2.087 (0.357)$ \\
$\text{SNR}=7$, AG &  $0.022 (0.004)$ &  $0.085 (0.014)$ &  $1.278 (0.262)$ \\
$\text{SNR}=7$, \texttt{ncvreg} &  $0.021 (0.004)$ &  $0.083 (0.015)$ &    $1.3 (0.262)$ \\
$\text{SNR}=10$, AG &  $0.016 (0.003)$ &  $0.065 (0.011)$ &  $1.163 (0.207)$ \\
$\text{SNR}=10$, \texttt{ncvreg} &  $0.015 (0.003)$ &  $0.063 (0.013)$ &   $1.167 (0.22)$ \\
\hline 
\hline 
\multicolumn{1}{c|}{PPV} & $\tau=0.1$ & $0.5$ & $0.9$\\
\hline 
$\text{SNR}=1$, AG &  $0.747 (0.134)$ &  $0.622 (0.188)$ &   $0.488 (0.25)$ \\
$\text{SNR}=1$, \texttt{ncvreg} &  $0.255 (0.061)$ &  $0.287 (0.132)$ &   $0.286 (0.19)$ \\
$\text{SNR}=3$, AG &  $0.681 (0.162)$ &  $0.551 (0.206)$ &  $0.327 (0.234)$ \\
$\text{SNR}=3$, \texttt{ncvreg} &  $0.282 (0.079)$ &  $0.307 (0.098)$ &  $0.275 (0.148)$ \\
$\text{SNR}=7$, AG &   $0.58 (0.138)$ &   $0.42 (0.257)$ &  $0.197 (0.141)$ \\
$\text{SNR}=7$, \texttt{ncvreg} &   $0.32 (0.065)$ &  $0.344 (0.152)$ &  $0.175 (0.101)$ \\
$\text{SNR}=10$, AG &  $0.528 (0.272)$ &   $0.437 (0.09)$ &  $0.211 (0.081)$ \\
$\text{SNR}=10$, \texttt{ncvreg} &  $0.349 (0.127)$ &    $0.409 (0.1)$ &  $0.206 (0.047)$ \\
\hline 
\hline 
\multicolumn{1}{c|}{NPV} & $\tau=0.1$ & $0.5$ & $0.9$\\
\hline 
$\text{SNR}=1$, AG &  $0.984 (0.001)$ &  $0.984 (0.001)$ &  $0.979 (0.001)$ \\
$\text{SNR}=1$, \texttt{ncvreg} &  $0.987 (0.001)$ &  $0.986 (0.001)$ &   $0.98 (0.001)$ \\
$\text{SNR}=3$, AG &  $0.989 (0.001)$ &  $0.988 (0.002)$ &   $0.98 (0.001)$ \\
$\text{SNR}=3$, \texttt{ncvreg} &   $0.99 (0.001)$ &  $0.989 (0.001)$ &   $0.98 (0.001)$ \\
$\text{SNR}=7$, AG &  $0.992 (0.001)$ &  $0.991 (0.001)$ &  $0.981 (0.001)$ \\
$\text{SNR}=7$, \texttt{ncvreg} &  $0.993 (0.001)$ &  $0.991 (0.001)$ &  $0.981 (0.001)$ \\
$\text{SNR}=10$, AG &  $0.993 (0.001)$ &  $0.992 (0.001)$ &  $0.982 (0.001)$ \\
$\text{SNR}=10$, \texttt{ncvreg} &  $0.993 (0.001)$ &  $0.992 (0.001)$ &  $0.982 (0.001)$ \\
\hline 
\hline 
\multicolumn{1}{c|}{$\lvert\hat{\mathcal{A}}\rvert$} & $\tau=0.1$ & $0.5$ & $0.9$\\
\hline 
$\text{SNR}=1$, AG &     $25.82 (8.08)$ &   $31.58 (17.056)$ &  $23.11 (15.166)$ \\
$\text{SNR}=1$, \texttt{ncvreg} &  $100.88 (25.582)$ &   $94.32 (41.572)$ &  $42.01 (20.592)$ \\
$\text{SNR}=3$, AG &   $42.78 (14.003)$ &   $55.48 (20.653)$ &  $42.83 (16.308)$ \\
$\text{SNR}=3$, \texttt{ncvreg} &  $120.17 (33.554)$ &  $101.75 (29.498)$ &  $46.72 (16.252)$ \\
$\text{SNR}=7$, AG &   $61.89 (21.881)$ &   $97.88 (36.736)$ &  $86.71 (26.567)$ \\
$\text{SNR}=7$, \texttt{ncvreg} &   $115.4 (23.845)$ &  $107.19 (31.445)$ &    $89.74 (23.1)$ \\
$\text{SNR}=10$, AG &  $101.21 (66.968)$ &   $81.17 (25.325)$ &   $70.8 (11.642)$ \\
$\text{SNR}=10$, \texttt{ncvreg} &   $123.5 (52.077)$ &   $90.58 (40.419)$ &  $71.47 (10.954)$ \\
\hline 
\end{tabular}

\end{center}
\end{table}

\newpage

\begin{table}[H]
\begin{center}

\caption{Signal recovery performance (sample mean and standard error of $\nicefrac{\left\Vert \boldsymbol{\beta}_{\text{true}}-\hat{\boldsymbol{\beta}}\right\Vert _{2}^{2}}{\left\Vert \boldsymbol{\beta}_{\text{true}}\right\Vert _{2}^{2}}$, Positive/Negative Predictive Values (PPV, NPV), and active set cardinality $\lvert\hat{\mathcal{A}}\rvert$ for signal detection) for {\tt ncvreg} and AG with our proposed hyperparameter settings on MCP-penalized linear model over $100$ simulation replications, across varying values of SNRs and covariates correlations ($\tau$). \label{tab:signal-lm-mcp}}
\begin{tabular}{c|ccc}
\hline 
\multicolumn{1}{c|}{$\nicefrac{\left\Vert \boldsymbol{\beta}_{\text{true}}-\hat{\boldsymbol{\beta}}\right\Vert _{2}^{2}}{\left\Vert \boldsymbol{\beta}_{\text{true}}\right\Vert _{2}^{2}}$} & $\tau=0.1$ & $0.5$ & $0.9$\\
\hline 
$\text{SNR}=1$, AG &  $0.133 (0.022)$ &  $0.563 (0.124)$ &   $2.839 (0.39)$ \\
$\text{SNR}=1$, \texttt{ncvreg} &  $0.126 (0.019)$ &  $0.494 (0.112)$ &   $2.86 (0.427)$ \\
$\text{SNR}=3$, AG &   $0.049 (0.01)$ &  $0.169 (0.034)$ &  $1.997 (0.329)$ \\
$\text{SNR}=3$, \texttt{ncvreg} &  $0.048 (0.009)$ &  $0.161 (0.032)$ &    $1.92 (0.34)$ \\
$\text{SNR}=7$, AG &  $0.021 (0.004)$ &  $0.088 (0.016)$ &  $1.503 (0.329)$ \\
$\text{SNR}=7$, \texttt{ncvreg} &   $0.02 (0.004)$ &  $0.086 (0.017)$ &  $1.416 (0.302)$ \\
$\text{SNR}=10$, AG &  $0.014 (0.003)$ &  $0.059 (0.011)$ &  $1.084 (0.272)$ \\
$\text{SNR}=10$, \texttt{ncvreg} &  $0.014 (0.003)$ &  $0.059 (0.013)$ &  $1.134 (0.248)$ \\
\hline 
\hline 
\multicolumn{1}{c|}{PPV} & $\tau=0.1$ & $0.5$ & $0.9$\\
\hline 
$\text{SNR}=1$, AG &   $0.85 (0.081)$ &  $0.744 (0.161)$ &  $0.616 (0.208)$ \\
$\text{SNR}=1$, \texttt{ncvreg} &  $0.435 (0.085)$ &  $0.407 (0.135)$ &  $0.387 (0.154)$ \\
$\text{SNR}=3$, AG &  $0.842 (0.119)$ &   $0.732 (0.21)$ &  $0.506 (0.286)$ \\
$\text{SNR}=3$, \texttt{ncvreg} &  $0.505 (0.112)$ &  $0.514 (0.121)$ &   $0.366 (0.18)$ \\
$\text{SNR}=7$, AG &  $0.761 (0.175)$ &  $0.646 (0.293)$ &  $0.505 (0.218)$ \\
$\text{SNR}=7$, \texttt{ncvreg} &  $0.541 (0.128)$ &  $0.547 (0.173)$ &  $0.483 (0.201)$ \\
$\text{SNR}=10$, AG &  $0.801 (0.099)$ &  $0.489 (0.134)$ &  $0.375 (0.225)$ \\
$\text{SNR}=10$, \texttt{ncvreg} &  $0.559 (0.107)$ &  $0.476 (0.135)$ &  $0.377 (0.225)$ \\
\hline 
\hline 
\multicolumn{1}{c|}{NPV} & $\tau=0.1$ & $0.5$ & $0.9$\\
\hline 
$\text{SNR}=1$, AG &  $0.983 (0.001)$ &  $0.982 (0.001)$ &  $0.979 (0.001)$ \\
$\text{SNR}=1$, \texttt{ncvreg} &  $0.986 (0.001)$ &  $0.984 (0.001)$ &    $0.979 (0.0)$ \\
$\text{SNR}=3$, AG &  $0.988 (0.001)$ &  $0.986 (0.001)$ &   $0.98 (0.001)$ \\
$\text{SNR}=3$, \texttt{ncvreg} &  $0.989 (0.001)$ &  $0.987 (0.001)$ &   $0.98 (0.001)$ \\
$\text{SNR}=7$, AG &  $0.991 (0.001)$ &  $0.989 (0.001)$ &  $0.981 (0.001)$ \\
$\text{SNR}=7$, \texttt{ncvreg} &  $0.992 (0.001)$ &  $0.989 (0.001)$ &  $0.981 (0.001)$ \\
$\text{SNR}=10$, AG &  $0.992 (0.001)$ &   $0.99 (0.001)$ &  $0.982 (0.001)$ \\
$\text{SNR}=10$, \texttt{ncvreg} &  $0.993 (0.001)$ &   $0.99 (0.001)$ &  $0.982 (0.001)$ \\
\hline 
\hline 
\multicolumn{1}{c|}{$\lvert\hat{\mathcal{A}}\rvert$} & $\tau=0.1$ & $0.5$ & $0.9$\\
\hline 
$\text{SNR}=1$, AG &    $19.7 (4.584)$ &     $20.6 (9.45)$ &    $12.5 (8.163)$ \\
$\text{SNR}=1$, \texttt{ncvreg} &  $51.61 (13.612)$ &  $47.32 (16.093)$ &  $20.25 (11.411)$ \\
$\text{SNR}=3$, AG &   $30.55 (8.437)$ &   $34.52 (16.44)$ &  $25.37 (14.373)$ \\
$\text{SNR}=3$, \texttt{ncvreg} &  $60.14 (15.873)$ &  $48.08 (13.783)$ &   $31.0 (13.981)$ \\
$\text{SNR}=7$, AG &  $44.45 (14.273)$ &  $56.95 (32.804)$ &  $31.96 (25.048)$ \\
$\text{SNR}=7$, \texttt{ncvreg} &   $66.7 (20.364)$ &  $58.36 (24.633)$ &  $33.38 (25.617)$ \\
$\text{SNR}=10$, AG &   $43.23 (11.26)$ &  $64.65 (12.923)$ &  $46.58 (18.186)$ \\
$\text{SNR}=10$, \texttt{ncvreg} &   $65.36 (13.06)$ &  $67.16 (15.483)$ &  $46.07 (19.223)$ \\
\hline 
\end{tabular}

\end{center}
\end{table}

\subsection{Penalized Logistic Regression
\label{sec:Simulations-logistic} }

Figure \eqref{fig:sim-ag-logistic-scad} and \eqref{fig:sim-ag-logistic-mcp}
suggest that much less iterations are needed for our method to achieve
the same amount of descent in comparison of AG with original proposed
settings for penalized logistic models.  

\begin{figure}[H]
\begin{centering}
\includegraphics[scale=1]{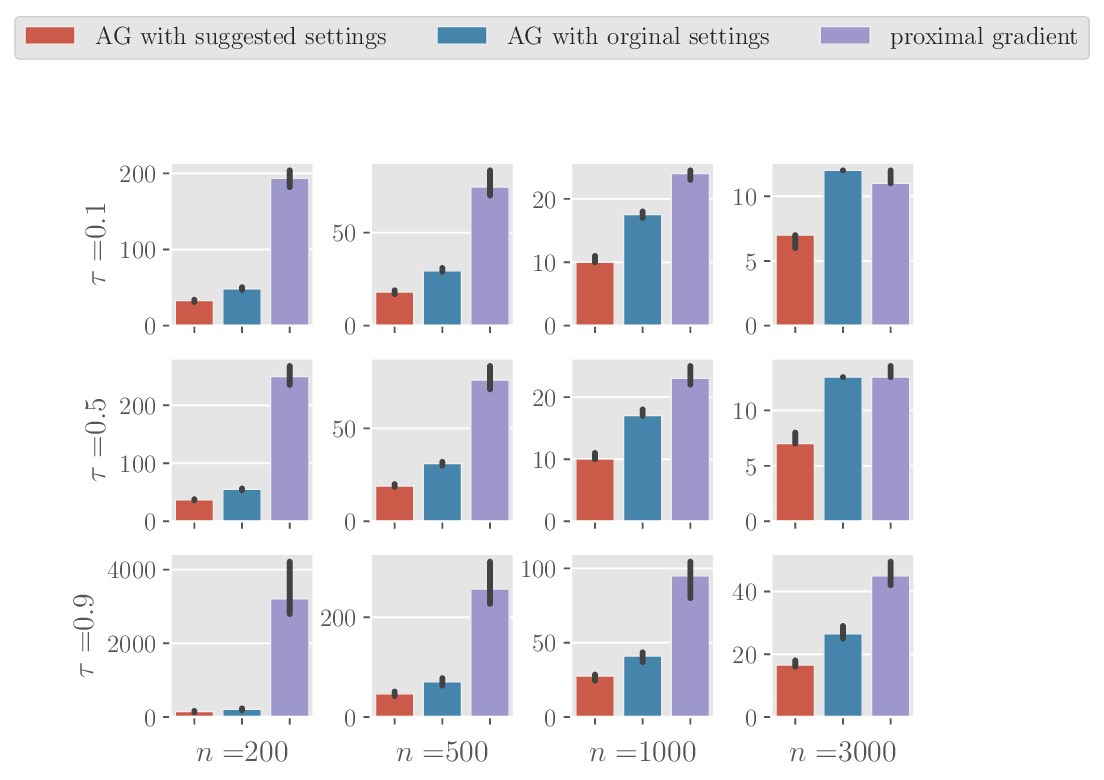}
\par\end{centering}
\caption{Median for the number of iterations required for the iterative objective values to reach $g^{*}+e^{2}$
on SCAD-penalized logistic regression for AG with our proposed hyperparameter
settings, AG with original settings, and proximal gradient over $100$ simulation replications, across varying covariates correlation ($\tau$) and $q/n$ values. The error bars represent the $95\%$ CIs from $1000$ bootstrap replications, $g^*$ represents the minimum per iterate found by the three methods considered.
\label{fig:sim-ag-logistic-scad}}
\end{figure}

\begin{figure}[H]
\begin{centering}
\includegraphics[scale=1]{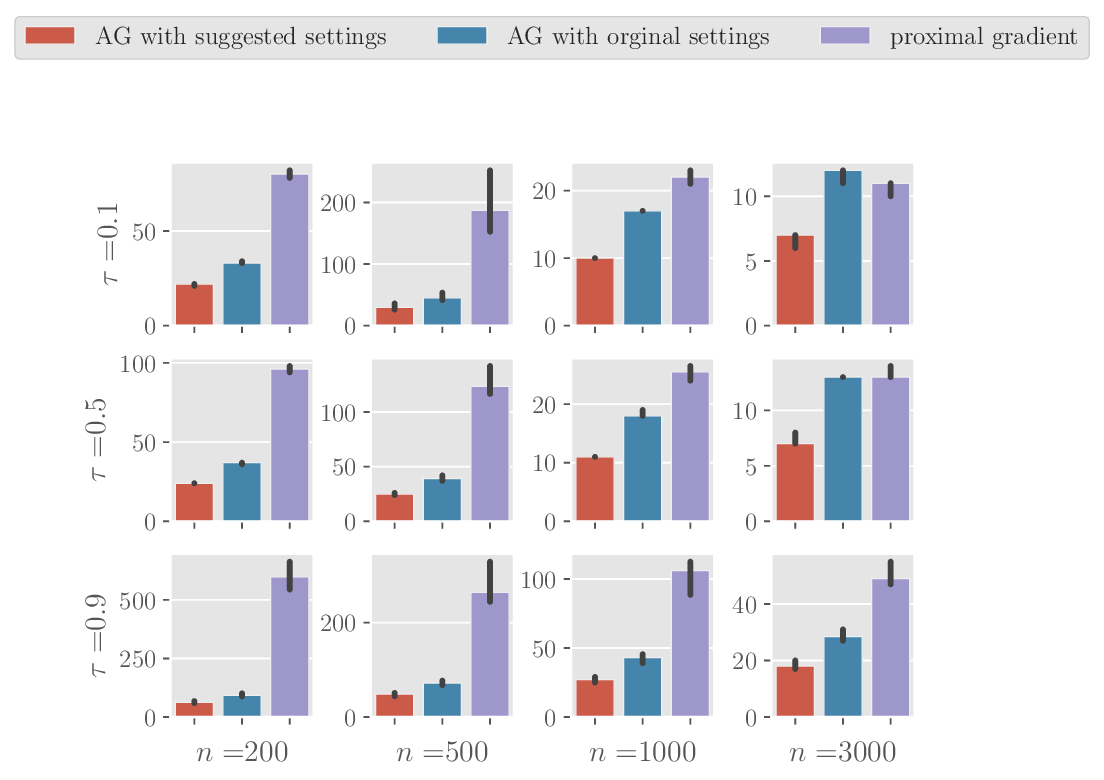}\caption{Median for the number of iterations required for iterative objective values to reach $g^{*}+e^{2}$
on MCP-penalized logistic regression for AG with our proposed hyperparameter
settings, AG with original settings, and proximal gradient over $100$ simulation replications, across varying covariates correlation ($\tau$) and $q/n$ values. The error bars represent the $95\%$ CIs from $1000$ bootstrap replications, $g^*$ represents the minimum per iterate found by the three methods considered.
\label{fig:sim-ag-logistic-mcp}}
\par\end{centering}
\end{figure}

\newpage

\begin{table}[H]
\begin{center}

\caption{Signal recovery performance (sample mean and standard error of $\nicefrac{\left\Vert \boldsymbol{\beta}_{\text{true}}-\hat{\boldsymbol{\beta}}\right\Vert _{2}^{2}}{\left\Vert \boldsymbol{\beta}_{\text{true}}\right\Vert _{2}^{2}}$, Positive/Negative Predictive Values (PPV, NPV), and active set cardinality $\lvert\hat{\mathcal{A}}\rvert$ for signal detection) for {\tt ncvreg} and AG with our proposed hyperparameter settings on SCAD-penalized logistic model over $100$ simulation replications, across varying values of SNRs and covariates correlations ($\tau$). \label{tab:signal-logistic-scad}}
\begin{tabular}{c|ccc}
\hline 
\multicolumn{1}{c|}{$\nicefrac{\left\Vert \boldsymbol{\beta}_{\text{true}}-\hat{\boldsymbol{\beta}}\right\Vert _{2}^{2}}{\left\Vert \boldsymbol{\beta}_{\text{true}}\right\Vert _{2}^{2}}$} & $\tau=0.1$ & $0.5$ & $0.9$\\
\hline 
$\text{SNR}=1$, AG &  $0.768 (0.047)$ &   $0.81 (0.041)$ &   $0.896 (0.04)$ \\
$\text{SNR}=1$, \texttt{ncvreg} &  $0.803 (0.033)$ &   $0.84 (0.033)$ &  $0.903 (0.037)$ \\
$\text{SNR}=3$, AG &  $0.556 (0.057)$ &  $0.656 (0.054)$ &  $0.839 (0.056)$ \\
$\text{SNR}=3$, \texttt{ncvreg} &  $0.603 (0.053)$ &  $0.682 (0.055)$ &  $0.813 (0.053)$ \\
$\text{SNR}=7$, AG &  $0.377 (0.076)$ &  $0.521 (0.073)$ &  $0.779 (0.072)$ \\
$\text{SNR}=7$, \texttt{ncvreg} &  $0.438 (0.054)$ &  $0.537 (0.074)$ &  $0.735 (0.074)$ \\
$\text{SNR}=10$, AG &  $0.311 (0.077)$ &  $0.474 (0.073)$ &  $0.757 (0.079)$ \\
$\text{SNR}=10$, \texttt{ncvreg} &  $0.377 (0.064)$ &  $0.481 (0.079)$ &  $0.712 (0.078)$ \\
\hline 
\hline 
\multicolumn{1}{c|}{PPV} & $\tau=0.1$ & $0.5$ & $0.9$\\
\hline 
$\text{SNR}=1$, AG &    $0.8 (0.079)$ &    $0.779 (0.1)$ &  $0.697 (0.126)$ \\
$\text{SNR}=1$, \texttt{ncvreg} &  $0.221 (0.045)$ &  $0.265 (0.079)$ &  $0.309 (0.169)$ \\
$\text{SNR}=3$, AG &  $0.875 (0.054)$ &  $0.859 (0.065)$ &  $0.765 (0.096)$ \\
$\text{SNR}=3$, \texttt{ncvreg} &  $0.244 (0.052)$ &  $0.273 (0.072)$ &  $0.273 (0.133)$ \\
$\text{SNR}=7$, AG &  $0.901 (0.052)$ &  $0.881 (0.057)$ &  $0.788 (0.098)$ \\
$\text{SNR}=7$, \texttt{ncvreg} &    $0.27 (0.04)$ &  $0.271 (0.079)$ &  $0.267 (0.136)$ \\
$\text{SNR}=10$, AG &  $0.915 (0.048)$ &  $0.899 (0.054)$ &  $0.789 (0.097)$ \\
$\text{SNR}=10$, \texttt{ncvreg} &    $0.29 (0.05)$ &  $0.279 (0.072)$ &   $0.26 (0.123)$ \\
\hline 
\hline 
\multicolumn{1}{c|}{NPV} & $\tau=0.1$ & $0.5$ & $0.9$\\
\hline 
$\text{SNR}=1$, AG &  $0.982 (0.001)$ &   $0.98 (0.001)$ &  $0.978 (0.001)$ \\
$\text{SNR}=1$, \texttt{ncvreg} &  $0.987 (0.002)$ &  $0.985 (0.002)$ &   $0.98 (0.001)$ \\
$\text{SNR}=3$, AG &  $0.985 (0.002)$ &  $0.982 (0.001)$ &  $0.979 (0.001)$ \\
$\text{SNR}=3$, \texttt{ncvreg} &   $0.99 (0.002)$ &  $0.987 (0.002)$ &   $0.98 (0.001)$ \\
$\text{SNR}=7$, AG &  $0.987 (0.002)$ &  $0.984 (0.001)$ &  $0.979 (0.001)$ \\
$\text{SNR}=7$, \texttt{ncvreg} &  $0.992 (0.001)$ &  $0.988 (0.001)$ &   $0.98 (0.001)$ \\
$\text{SNR}=10$, AG &  $0.988 (0.002)$ &  $0.984 (0.001)$ &  $0.979 (0.001)$ \\
$\text{SNR}=10$, \texttt{ncvreg} &  $0.992 (0.001)$ &  $0.988 (0.001)$ &   $0.98 (0.001)$ \\
\hline 
\hline 
\multicolumn{1}{c|}{$\lvert\hat{\mathcal{A}}\rvert$} & $\tau=0.1$ & $0.5$ & $0.9$\\
\hline 
$\text{SNR}=1$, AG &     $17.07 (3.91)$ &     $13.4 (3.365)$ &    $7.62 (2.134)$ \\
$\text{SNR}=1$, \texttt{ncvreg} &  $120.14 (28.882)$ &   $86.49 (24.421)$ &  $39.41 (19.448)$ \\
$\text{SNR}=3$, AG &    $23.34 (4.203)$ &    $16.59 (3.459)$ &    $8.69 (2.082)$ \\
$\text{SNR}=3$, \texttt{ncvreg} &   $134.85 (29.96)$ &   $98.48 (28.434)$ &  $42.47 (15.014)$ \\
$\text{SNR}=7$, AG &     $26.98 (4.58)$ &    $19.46 (3.659)$ &    $9.79 (2.246)$ \\
$\text{SNR}=7$, \texttt{ncvreg} &  $130.33 (22.255)$ &  $105.03 (28.123)$ &  $48.81 (19.059)$ \\
$\text{SNR}=10$, AG &    $27.95 (4.462)$ &    $19.57 (3.141)$ &   $10.24 (2.346)$ \\
$\text{SNR}=10$, \texttt{ncvreg} &  $124.58 (23.016)$ &   $103.49 (27.66)$ &  $50.64 (21.138)$ \\
\hline 
\end{tabular}

\end{center}
\end{table}

\newpage

\begin{table}[H]
\begin{center}

\caption{Signal recovery performance (sample mean and standard error of $\nicefrac{\left\Vert \boldsymbol{\beta}_{\text{true}}-\hat{\boldsymbol{\beta}}\right\Vert _{2}^{2}}{\left\Vert \boldsymbol{\beta}_{\text{true}}\right\Vert _{2}^{2}}$, Positive/Negative Predictive Values (PPV, NPV), and active set cardinality $\lvert\hat{\mathcal{A}}\rvert$ for signal detection) for {\tt ncvreg} and AG with our proposed hyperparameter settings on MCP-penalized logistic model over $100$ simulation replications, across varying values of SNRs and covariates correlations ($\tau$). \label{tab:signal-logistic-mcp}}
\begin{tabular}{c|ccc}
\hline 
\multicolumn{1}{c|}{$\nicefrac{\left\Vert \boldsymbol{\beta}_{\text{true}}-\hat{\boldsymbol{\beta}}\right\Vert _{2}^{2}}{\left\Vert \boldsymbol{\beta}_{\text{true}}\right\Vert _{2}^{2}}$} & $\tau=0.1$ & $0.5$ & $0.9$\\
\hline 
$\text{SNR}=1$, AG &  $0.769 (0.044)$ &  $0.808 (0.041)$ &  $0.897 (0.043)$ \\
$\text{SNR}=1$, \texttt{ncvreg} &  $0.795 (0.036)$ &  $0.829 (0.032)$ &  $0.903 (0.038)$ \\
$\text{SNR}=3$, AG &  $0.555 (0.058)$ &  $0.654 (0.053)$ &  $0.834 (0.054)$ \\
$\text{SNR}=3$, \texttt{ncvreg} &  $0.605 (0.049)$ &  $0.674 (0.054)$ &  $0.825 (0.057)$ \\
$\text{SNR}=7$, AG &   $0.383 (0.08)$ &  $0.521 (0.069)$ &   $0.779 (0.07)$ \\
$\text{SNR}=7$, \texttt{ncvreg} &  $0.438 (0.057)$ &   $0.533 (0.07)$ &  $0.761 (0.071)$ \\
$\text{SNR}=10$, AG &   $0.31 (0.079)$ &  $0.469 (0.073)$ &  $0.753 (0.076)$ \\
$\text{SNR}=10$, \texttt{ncvreg} &  $0.381 (0.061)$ &   $0.48 (0.082)$ &  $0.737 (0.077)$ \\
\hline 
\hline 
\multicolumn{1}{c|}{PPV} & $\tau=0.1$ & $0.5$ & $0.9$\\
\hline 
$\text{SNR}=1$, AG &   $0.879 (0.06)$ &  $0.859 (0.058)$ &  $0.779 (0.087)$ \\
$\text{SNR}=1$, \texttt{ncvreg} &  $0.372 (0.068)$ &  $0.401 (0.106)$ &  $0.375 (0.157)$ \\
$\text{SNR}=3$, AG &   $0.906 (0.05)$ &   $0.889 (0.05)$ &  $0.805 (0.086)$ \\
$\text{SNR}=3$, \texttt{ncvreg} &   $0.43 (0.065)$ &  $0.445 (0.106)$ &  $0.395 (0.126)$ \\
$\text{SNR}=7$, AG &  $0.919 (0.044)$ &   $0.903 (0.05)$ &  $0.809 (0.102)$ \\
$\text{SNR}=7$, \texttt{ncvreg} &  $0.463 (0.063)$ &   $0.45 (0.104)$ &  $0.417 (0.145)$ \\
$\text{SNR}=10$, AG &  $0.918 (0.045)$ &  $0.911 (0.038)$ &  $0.804 (0.111)$ \\
$\text{SNR}=10$, \texttt{ncvreg} &  $0.502 (0.069)$ &  $0.468 (0.095)$ &  $0.412 (0.137)$ \\
\hline 
\hline 
\multicolumn{1}{c|}{NPV} & $\tau=0.1$ & $0.5$ & $0.9$\\
\hline 
$\text{SNR}=1$, AG &  $0.981 (0.001)$ &   $0.98 (0.001)$ &  $0.978 (0.001)$ \\
$\text{SNR}=1$, \texttt{ncvreg} &  $0.986 (0.002)$ &  $0.983 (0.001)$ &  $0.978 (0.001)$ \\
$\text{SNR}=3$, AG &  $0.985 (0.002)$ &  $0.982 (0.001)$ &  $0.979 (0.001)$ \\
$\text{SNR}=3$, \texttt{ncvreg} &  $0.989 (0.002)$ &  $0.985 (0.001)$ &  $0.979 (0.001)$ \\
$\text{SNR}=7$, AG &  $0.987 (0.002)$ &  $0.984 (0.001)$ &   $0.98 (0.001)$ \\
$\text{SNR}=7$, \texttt{ncvreg} &  $0.991 (0.002)$ &  $0.986 (0.001)$ &   $0.98 (0.001)$ \\
$\text{SNR}=10$, AG &  $0.988 (0.002)$ &  $0.984 (0.001)$ &   $0.98 (0.001)$ \\
$\text{SNR}=10$, \texttt{ncvreg} &  $0.991 (0.001)$ &  $0.987 (0.001)$ &   $0.98 (0.001)$ \\
\hline 
\end{tabular}
\begin{tabular}{c|ccc}
\hline 
\multicolumn{1}{c|}{$\lvert\hat{\mathcal{A}}\rvert$} & $\tau=0.1$ & $0.5$ & $0.9$\\
\hline 
$\text{SNR}=1$, AG &   $13.86 (3.082)$ &   $11.42 (2.776)$ &   $6.72 (1.744)$ \\
$\text{SNR}=1$, \texttt{ncvreg} &  $59.83 (14.138)$ &   $42.1 (12.546)$ &  $19.72 (8.393)$ \\
$\text{SNR}=3$, AG &   $21.86 (4.313)$ &   $15.84 (3.036)$ &   $8.84 (1.938)$ \\
$\text{SNR}=3$, \texttt{ncvreg} &  $66.57 (13.203)$ &    $48.28 (14.5)$ &  $22.81 (9.784)$ \\
$\text{SNR}=7$, AG &   $25.75 (4.776)$ &   $18.78 (3.189)$ &  $10.33 (2.565)$ \\
$\text{SNR}=7$, \texttt{ncvreg} &  $69.44 (11.876)$ &  $52.54 (13.638)$ &  $24.63 (8.741)$ \\
$\text{SNR}=10$, AG &   $27.53 (4.649)$ &   $19.55 (3.093)$ &  $11.06 (2.877)$ \\
$\text{SNR}=10$, \texttt{ncvreg} &  $65.38 (10.776)$ &  $51.66 (12.785)$ &  $25.59 (9.428)$ \\
\hline 
\end{tabular}

\end{center}
\end{table}

\newpage

\printbibliography

\end{document}